\documentclass[a4paper,11pt,oneside,reqno]{amsart}
\usepackage{amsmath,amsfonts,amstext,amssymb,amsbsy,amsopn,amsthm,eucal,amsrefs}
\usepackage{vmargin}
\usepackage[pdftex]{graphicx}
\usepackage{mathtools}
\usepackage{amssymb}
\usepackage{mathrsfs}
\usepackage{stmaryrd}
\usepackage{amsmath}
\usepackage{amsthm}
\usepackage{enumerate}
\usepackage{float}
\usepackage[usenames,dvipsnames]{xcolor}
\usepackage{nomencl}
\usepackage[bookmarks=true]{hyperref}   
\hypersetup{
    colorlinks = true,
    citecolor=Mahogany,
    filecolor= black,
    linkcolor=Plum,
    urlcolor=Aquamarine
}
\usepackage{tikz-cd}
\usepackage{setspace}

\usepackage{soul}

\usepackage{times}

\usepackage[T1]{fontenc}

\usepackage[active]{srcltx}

\usepackage{enumitem}

\makeatletter
\renewcommand\section{\@startsection{section}{1}{\z@}%
                       {-3\p@ \@plus -4\p@ \@minus -4\p@}%
                       {3\p@ \@plus 4\p@ \@minus 4\p@}%
                      {\normalfont\normalsize\centering\scshape}}
\makeatother

\hypersetup{
    colorlinks,%
}

\author{Hsuan-Yi Liao \textsuperscript{1}}
\author{Zachary McGuirk \textsuperscript{2}}
\author{Dang Khoa Nguyen \textsuperscript{3}}
\author{Byungdo Park \textsuperscript{4}}

\address{\textsuperscript{1} \small Hsuan-Yi Liao,
	Department of Mathematics, National Tsing Hua University, Hsinchu 300, Taiwan} 
\email{\href{mailto:hyliao@math.nthu.edu}{hyliao@math.nthu.edu}}

\address{\textsuperscript{2} \small Zachary McGuirk, Einstein Institute of Mathematics, Edmond J. Safra Campus, The Hebrew University of Jerusalem, Jerusalem 91904, Israel} 

\curraddr{College of Arts and Sciences, New York Institute of Technology, 16 West 61st St, New York, NY 10023, USA}
\email{\href{mailto:zmcguirk@nyit.edu}{zmcguirk@nyit.edu}}

\address{\textsuperscript{3} \small Dang Khoa Nguyen,
	Department of Mathematics, National Tsing Hua University, Hsinchu 300, Taiwan} 
\email{\href{mailto:nguyendangkhoa211@gmail.com}{nguyendangkhoa211@gmail.com}}
    
\address{\textsuperscript{4} \small Byungdo Park,
	Department of Mathematics Education, Chungbuk National University, Cheongju 28644, Republic of Korea} 
\email{\href{mailto:byungdo@cbnu.ac.kr}{byungdo@cbnu.ac.kr}}
\thanks{This work was partially supported by the National Research Foundation of Korea (NRF) grant funded by the Korea government (MSIT) (No. 2020R1G1A1A01008746), Chungbuk National University NUDP program (2024), and the MoST/NSTC Grants 110-2115-M-007-001-MY2 and 112-2115-M-007-016-MY3.}

\makeatletter
\@namedef{subjclassname@2020}{%
  \textup{2020} Mathematics Subject Classification}
\makeatother

\newcommand{\cC}{\mathcal{C}} 
\newcommand{\cD}{\mathcal{D}}

\DeclareMathOperator{\image}{im}

\newcommand{\RNum}[1]{\uppercase\expandafter{\romannumeral #1\relax}}
\usepackage{fancyhdr}
\fancypagestyle{plain}{
    \fancyhf{} 
    \fancyfoot[C]{\thepage} 
                                \fancyhead[LO]{\small Brown functors of directed graphs}
    \fancyhead[RO]{\small }
}

\def\colour{\colour}

\def\colour{\color}

\theoremstyle{definition}    \newtheorem{dfn}{Definition}[section]
\theoremstyle{definition}    \newtheorem{rmk}[dfn]{Remark}
\theoremstyle{definition}    
\theoremstyle{definition}    \newtheorem{prp}[dfn]{Proposition}
\theoremstyle{definition}    
\theoremstyle{definition}    \newtheorem{nta}[dfn]{Notation}
\theoremstyle{definition}    \newtheorem{thm}[dfn]{Theorem}
\theoremstyle{definition}    
\theoremstyle{definition}    \newtheorem{lem}[dfn]{Lemma}
\theoremstyle{definition}    
\theoremstyle{definition}    \newtheorem{exa}[dfn]{Example}         
\theoremstyle{definition}    
\theoremstyle{definition}    
\theoremstyle{definition}    

\newcommand{\e}{\mathrm{e}}

\newcommand{\id}{{\rm id}}
\renewcommand{\epsilon}{\varepsilon}

\renewcommand{\phi}{\varphi}

\newcommand{\embed}{\hookrightarrow}		

\DeclareFontFamily{OT1}{restrictfont}{}
\DeclareFontShape{OT1}{restrictfont}{m}{n}{<-> fmvr8x}{}

\usepackage[all]{xy}
\usepackage{color,xcolor}

\newcommand{\Z}{{\mathbb Z}}

\newcommand{\cA}{\mathcal A}

\newcommand{\cR}{\mathcal R}
\newcommand{\f}{\mathfrak}

\newcommand{\al}{\alpha}
\newcommand{\be}{\beta}
\newcommand{\ga}{\gamma}     \newcommand{\Ga}{{\Gamma}}
\newcommand{\de}{\delta}

    \newcommand{\La}{{\Lambda}}
\newcommand{\te}{\theta}    \newcommand{\Te}{{\Theta}}

  \newcommand{\Hom}{\textrm{Hom}}

\newcommand{\hD}{\text{Ho}\mathcal{D}}
\newcommand{\whD}{\text{wHo}\mathcal{D}}

\newcommand{\srl}{\stackrel}
 \newcommand{\ra}{\rightarrow}

\newcommand{\isom}{\cong}     

\newcommand{\bmat}{\left(\begin{array}}  \newcommand{\emat}{\end{array}\right)}
\newcommand{\barr}{\begin{array}}  \newcommand{\earr}{\end{array}}
\newcommand{\bcd}{\begin{CD}}  \newcommand{\ecd}{\end{CD}}
\newcommand{\beq}{\begin{equation}\begin{aligned}}  \newcommand{\eeq}{\end{aligned}\end{equation}}
\newcommand{\beqs}{\begin{equation*}\begin{aligned}}  \newcommand{\eeqs}{\end{aligned}\end{equation*}}
  
\newcommand{\vG}{\vec{G}} 
\newcommand{\vH}{\vec{H}}

\newcommand{\fsub}{\underset{\text{finite}}{\subseteq}}
\newcommand{\BF}{\mathrm{H}}
\newcommand{\EBF}{\widehat{\BF}}
\newcommand{\Nat}{\mathrm{Nat}}

\begin{document}
	
	\title{Brown functors of directed graphs}
	\date{22 June 2025}
	
	\maketitle

	\parindent0cm
	\setlength{\parskip}{\baselineskip}

\begin{abstract}
    We prove that any digraph Brown functor --- i.e.\ a contravariant functor from the homotopy category of finite directed graphs to the category of abelian groups, satisfying the triviality axiom, the additivity axiom, and the Mayer-Vietoris axiom --- is representable. Furthermore, we show that the first path cohomology functor is a digraph Brown functor. 
\end{abstract}

\section{Introduction}
The homotopy theory of directed graphs is a discrete analogue of homotopy theory in algebraic topology. In topology, a homotopy between two continuous maps is defined as a family of continuous maps parametrized by the closed interval $[0,1]$, providing a continuous interpolation between them. Its discrete counterpart, which we study, uses a directed line graph to track discrete changes in directed graph maps.

Efforts to develop a homotopy theory for graphs date back to the 1970s and 1980s, with early contributions by Gianella~\cite{G76} and Malle~\cite{M83}. However, the modern formulation of graph homotopy theory gained momentum with a 2001 paper by Chen, Yau, and Yeh~\cite{BYY01}, culminating in the 2014 work of Grigor'yan, Lin, Muranov, and Yau~\cite{GLMY14}. More recently, interest in this notion of homotopy for directed graphs has grown, particularly following a result by Grigor'yan, Jimenez, Muranov, and Yau in~\cite{GJMY18}, which demonstrated that the path homology theory of directed graphs satisfies a discrete version of the Eilenberg-Steenrod axioms.

In this paper, we investigate digraph Brown functors, which are contravariant functors from the homotopy category of finite directed graphs to the category of abelian groups, satisfying the triviality axiom, the additivity axiom, and the Mayer-Vietoris axiom (Definition~\ref{DFN.brown.functor}). 
Our main result is that any digraph Brown functor is representable. That is, if $\BF$ is a digraph Brown functor, then there exist a classifying directed graph $\vec{Y}$ (which is not necessarily finite) and a natural isomorphism between the functor $[-,\vec{Y}]$, which assigns homotopy classes of maps, and the digraph Brown functor $\BF(-)$. See Theorem~\ref{THM.main.theorem}.

The Brown representability theorem is a classical theorem in algebraic topology first proved by Edgar H. Brown~\cite{B62}. It states that a $\textbf{Set}$-valued functor on the homotopy category of based CW complexes, satisfying the wedge axiom and the Mayer-Vietoris axiom, is representable. Brown went further in~\cite{B65} by replacing the homotopy category of based CW complexes with an arbitrary category satisfying a list of proposed axioms, and this result has been further generalized in triangulated categories by Neeman~\cite{Ne}, closed model categories in Jardine~\cite[Theorem 19]{Ja}, and homotopy categories of $\infty$-categories in Lurie~\cite[Section 1.4.1]{Lu}. The gist of these generalizations is that Brown representability is more of a category-theoretic feature than a topological feature. However, the idea behind the classical theorem of J. H. C. Whitehead that every CW complex is formed by attaching spheres is essential, whereas an analogue of attaching spheres in the homotopy category of directed graphs is not well-understood yet. A more fundamental issue in here is that the homotopy extension property fails in the category of directed graphs (see Remark~\ref{EXA.HEP.does.not.hold}).
In particular, this is why a straightforward verification of Brown's axioms~\cite{B65} for the category of directed graphs is insufficient to establish this representability result for directed graphs.

Our approach to the Brown representability theorem for directed graphs is mainly inspired by Adams~\cite{Ad}. However, adapting Adams' method to directed graphs requires developing new combinatorial tools that serve as discrete analogues of Adams' constructions while respecting inherent limitations of digraph homotopy theory. Specifically, we construct a classifying directed graph $\vec{Y}$ for a Brown functor $\BF$ by attaching all possible mapping tubes (see Definition~\ref{DFN.mapping.tube}) to an infinite directed graph \eqref{eq:Y1}. See Lemma~\ref{induct} and Lemma~\ref{LEM.representability.on.finite.digraph}. In this process, we extend $\BF$ to a functor on the category of arbitrary (including infinite) directed graphs while ensuring that it still satisfies a version of the Mayer-Vietoris axiom. This extension is achieved by considering modified mapping cones (Definition~\ref{DFN.digraph.cofiber}). The key properties of a modified mapping cone $\vec{C}(f)$ for a map $f:\vG \to \vH$ of directed graphs are as follows:
\begin{itemize}
    \item[(i)] There is a natural embedding $\vG \hookrightarrow \vec{C}(f)$ that is homotopic to both a constant map and the map $\vG \to \vec{C}(f)$ induced by $f$ (see Example~\ref{ex:CofiberHpt}).
    \item[(ii)] There is an associated exact sequence $\BF(\vec{C}(f)) \to \BF(\vH) \to \BF(\vG)$ for a Brown functor $\BF$ (see Lemma~\ref{LEM.adams.3.1}).
    \item[(iii)] If $f$ is the map $\vG \amalg \vH \to \vG \cup \vH$, then $\vec{C}(f)$ is homotopic to the subdigraph generated by the intersection $\vG \cap \vH$ (see Figure~\ref{fig:suspension}).
\end{itemize}
These properties guarantee the existence of such an extended functor, as demonstrated in Proposition~\ref{PRP.Adams.3.5}, where we employ an alternative construction $S$ to compensate for the absence of natural suspension operations for this purpose.

The Brown representability theorem in topology requires the domain category of the functors to be the category of connected CW complexes with a base point. When the domain category is either the category of not-necessarily-connected CW complexes or the category of unbased CW complexes, well-known counterexamples exist; see, for example, Heller~\cite{He} and Freyd and Heller~\cite{FH}. The category of directed graphs is not as rich as the category of topological spaces, and the Freyd–Heller counterexample does not occur.

As a concrete example of digraph Brown functors, we consider path cohomology of directed graphs \cite{GLMY15,GMY16,GMY17}. Explicitly, we prove that the first path cohomology is a digraph Brown functor.

This notion of path (co)homology is particularly interesting because counting the essential types of cycles in a directed graph is relevant in various contexts where directed graphs model physical and real-world phenomena, such as quivers \cites{Ci, Ho, GS}, neural networks \cite{CGHY}, electrical circuits \cite{DSB}, and gauge theory \cite{DM}. Additionally, it has applications in homotopy theory \cite{CDKOSW}, persistent homology \cite{CHT}, and Hochschild cohomology of algebras~\cite{GMY16}.  Recently, another version of directed graph homology theory that satisfies all invariance, Kunneth, excision, and Mayer-Vietoris theorems was considered \cite{HR}. 
Furthermore, since Hochschild cohomology is equipped a structure of Gerstenhaber algebra \cite{G63}, with the connection between path cohomology and Hochschild cohomology \cite{GMY16}, it would be interesting to construct a transfered homotopy Gerstenhaber algebra structure on a digraph, and explore a digraph version of relevant theorems \cite{LS23,LSX18}.

The paper is organized as follows. Section~\ref{SEC.section.2} provides background on digraph homotopy theory and the relevant constructions. It serves a twofold purpose: establishing notation and conventions, and presenting constructions that underpin our arguments and main results in later sections.  
Section~\ref{SEC.section.3} defines a Brown functor, introduces its extended version, and discusses their properties. Section~\ref{SEC.section.4} constructs a classifying directed graph of a Brown functor. In Section~\ref{sec:PathCoh}, we consider the path cohomology of directed graphs and prove that the first path cohomology is a Brown functor.

\textbf{Acknowledgements.} We thank J\'ozef Dodziuk and Martin Bendersky for their interest on this work. Additionally, we thank Matthew Burfitt, Seokbong Seol, Gr\'{e}goire Sergeant-Perthuis, Jyh-Haur Teh, Kuang-Ru Wu, Chi Ho Yuen, Ping Xu and Tom Zaslavsky for helpful comments. Some of the research which resulted in this paper was carried out in Seoul at Korea Institute for Advanced Study (KIAS) and at the Hebrew University of Jerusalem. We thank these institutions for their support and hospitality.

\section{Homotopy theory for directed graphs}\label{SEC.section.2}
In this section, we shall give a brief review of directed graph homotopy theory as well as relevant constructions. A good reference on digraph homotopy theory is Grigor'yan, Lin, Muranov, and Yau~\cite{GLMY14} which has a broader account. Several constructions we give in this section are the technical core of this paper. See ~\ref{DFN.digraph.cofiber} and~\ref{DFN.mapping.tube}.

    \subsection{The category of directed graphs}
   A \textbf{directed graph} (or \textbf{digraph} for short) $\vec{G}$ is a pair $(V,E)$ consisting of a set $V$ specifying labeled points called \emph{vertices} and another set $E$ of ordered pairs of distinct vertices in $V$ called \emph{edges}. Having an edge $(x,y)\in E$ means that there is a directed arrow from $x$ to $y$ and graphically one draws $ x\rightarrow y$. Note that from the definition above, loop-edges are excluded from consideration and since $E$ is a set, $(x,y)$ occurs at most once. When multiple digraphs are involved in the context, we often denote the set of vertices of $\vG$ by $V_{\vG}$ and the set of edges by $E_{\vG}$.

A \textbf{point} $\ast$ is a digraph consisting of only one vertex and no edges, and an \textbf{$n$-step line digraph}, $n>0$, is a sequence of vertices, $0$, $1$, $2$,$\ldots$, $n$, such that either $(i-1,i)$ or $(i,i-1)$, for $1\leq i\leq n$, is an edge (but not both) and there are no other edges. Note that an $n$-step line digraph is also called \emph{path digraph} or a \emph{linear digraph}. Such a directed graph forms a line with $n$ arbitrarily oriented edges between each of the $n+1$ vertices. When $n=1$, there are two possible line digraphs, $I^+:= 0\rightarrow 1$ and $I^-:= 0\leftarrow 1$. We will denote an arbitrary $n$-step line digraph as $I_n$ for short and let $\mathcal{I}_n$ represent the set of all possible $n$-step line digraphs. The set of all line digraphs of any length  will be denoted $\mathcal{I}=\bigcup_n\mathcal{I}_n$ and we will refer to an arbitrary element of $\mathcal{I}$ as a line digraph $I$, dropping the reference to the number of steps.

A \textbf{digraph map}, $f\colon\vec{G}\to\vec{H}$, is a function from the vertex set of $\vec{G}$ to the vertex set of $\vec{H}$ such that whenever $(x,y)$ is an edge in $\vec{G}$ either $f(x)=f(y)$ in $\vec{H}$ or $ {f(x)}\rightarrow {f(y)}$ is an edge in $\vec{H}$. We denote by $\image(f) = f(\vG)$ the image of $f$ which is a digraph. If for some edge $(x,y)$, $f(x)=f(y)$ in $\vec{H}$, then we will say that this edge has been collapsed and if $(f(x),f(y))\in E_{\vec{H}}$, then we say that the edge has been preserved. Since a digraph map must be a function on the discrete set of vertices, the image of a digraph map has at most as many vertices as the domain.
 
The \textbf{category of directed graphs} $\cD$ is the category whose objects are directed graphs, $\vec{G}$, and the morphisms are digraph maps, $f\colon\vec{G}\to\vec{H}$. A graph $\vG=(V,E)$ is \textbf{finite} if the vertex set $V$ is finite. We will use the notation $\mathcal{D}_0$ to denote the category whose objects are finite digraphs and morphisms are digraph maps. The category $\cD_0$ is a subcategory of $\cD$. Throughout this paper, the expression $X\in\cC$ for a category $\cC$ means $X$ is an object of the category $\cC$. We will write $f\in\cC(X,Y)$ to say $f$ is a morphism from $X$ to $Y$ in $\cC$.
   
    \subsection{Operations for directed graphs}\label{sec:Operation}

A \textbf{subdigraph} $\vec{X}$ of a digraph $\vec{G}$ denoted $\vec{X}\subseteq \vec{G}$ is a digraph for which $V_{\vec{X}}\subseteq V_{\vec{G}}$ and $E_{\vec{X}}\subseteq E_{\vec{G}}$. Note that even if $u,v\in V_{\vec{X}}$ and $(u,v)\in E_{\vec{G}}$, it is not necessarily the case that $(u,v)\in E_{\vec{X}}$. An \textbf{induced} subdigraph $\vec{X}$ of a digraph $\vec{G}$ denoted $\vec{X}\sqsubset\vec{G}$ is a subdigraph in which whenever $u,v\in V_{\vec{X}}$ and $(u,v)\in E_{\vec{G}}$, then $(u,v)\in E_{\vec{X}}$ too.

Let $\vG$ and $\vH$ be subdigraphs of a digraph $\vec Y$. The \textbf{intersection} of digraphs $\vec{G}$ and $\vec{H}$, denoted by $\vec{G}\cap\vec{H}$, is the digraph consisting of $V_{\vec{G}\cap\vec{H}}=V_{\vec{G}}\cap V_{\vec{H}}$ and $E_{\vec{G}\cap\vec{H}}=E_{\vec{G}}\cap E_{\vec{H}}$. Note that $\vec{G}\cap\vec{H}$ is not necessarily an induced subdigraph of $\vec{G}$ and $\vec{H}$. The \textbf{union} of digraphs $\vec{G}$ and $\vec{H}$, denoted by $\vec{G}\cup\vec{H}$, is the digraph consisting of $V_{\vec{G}\cup\vec{H}}=V_{\vec{G}}\cup V_{\vec{H}}$ and $E_{\vec{G}\cup\vec{H}}=E_{\vec{G}}\cup E_{\vec{H}}$. Note that $\vec{G}$ and $\vec{H}$ are necessarily induced subdigraphs of $\vec{G}\cup\vec{H}$. 
Note that the \textbf{disjoint union} of two digraphs $\vG$ and $\vec{H}$, denoted $\vG\coprod\vec{H}$ is given by the disjoint union of their respective vertex sets and edge sets, as sets. The disjoint union is the coproduct in the category $\cD$. 
    
The \textbf{graph Cartesian product} of two directed graphs $\vec{G}$ and $\vec{H}$ is the directed graph $\vec{G}\Box\vec{H}$, where the vertices are all ordered pairs $(u,v)$ such that $u\in V_{\vec{G}}$ and $v\in V_{\vec{H}}$, and $ {(u_1,v_1)}\rightarrow {(u_2,v_2)}$ is an edge in $\vec{G}\Box\vec{H}$ if either $u_1=u_2$ and $v_1\ra v_2$ in $\vec{H}$, or $u_1\ra u_2$ in $\vec{G}$ and $v_1=v_2$. Note that the graph Cartesian product is not a product in the category $\cD$.
    Given a fixed vertex $v_0\in V_{\vec{H}}$, we will denote by $\vec{G}\Box\{v_0\}$ the $v_0$-slice of $\vec{G}\Box\vec{H}$. It is the induced subdigraph where the vertices are all ordered pairs $(u,v_0)$ such that $u\in V_{\vec{G}}$ and the edges are those resulting from the edges of $\vG$.


        Let $\sim$ be an equivalence relation on the vertex set of a digraph $\vG$. The equivalence classes naturally form a digraph 
        $$
        \vG / \sim \quad  = \quad  (\tilde V, \tilde E),
        $$
        where $\tilde V = V/\sim$ is the set of equivalence classes of vertices, and $(x,y) \in \tilde E \subseteq \tilde V \times \tilde V$ if and only if $x \neq y$ and there exist $u \in x$, $v \in y$ such that $(u,v) \in E_{\vG}$. A \textbf{quotient digraph} $\vec{G}/\vec{X}$, for $\vec{X}\subseteq\vec{G}$ and $\vec{X}$ not necessarily connected, is the digraph $\vG\coprod\ast/\sim$ where $x\sim \ast$ for all $x\in V_{\vec{X}}$. The \textbf{mapping cylinder of a digraph map} $f\colon\vec{G}\to\vec{H}$ is given by $$\vec{M}_f:=\Big(\vec{G}\Box I^-\coprod \vec{H}\Big)\Big/ \sim,$$ where $V_{\vG \Box I^-} \ni (g,0)\sim f(g) \in \vH$ for each $g\in V_{\vec{G}}$.  The \textbf{cone} over a digraph $\vec{G}$, denoted by $C\vec{G}$, is the digraph $\big(\vec{G}\Box I^- \coprod \ast \big) \big/ \sim$, where $(g,1)\sim\ast$ for all $ g\in V_{\vec{G}}.$

\begin{dfn}
        For a digraph map $f: \vG \to \vH$, consider the set $V_{\image_2(f)} = \{h \in V_{\vH} : |f^{-1}(h)| \geq 2\}$. 
    A \textbf{modified mapping cylinder} $\widehat{M}_f$ is defined by $V_{\widehat{M}_f} = V_{M_f}$ and 
    \begin{equation}
        E_{\widehat{M}_f} = E_{M_f} \coprod \dot E_{\widehat{M}_f}  \coprod \ddot E_{\widehat{M}_f},
    \end{equation}
    where 
    \begin{align*}
         \dot E_{\widehat{M}_f} & = \{ \big((g,1),h'\big)  : (f(g),h')  \in E_{\image(f)}, h' \in V_{\image_2(f)}, g \in V_{\vG}, f(g) \notin  V_{\image_2(f)} \} , \\
         \ddot E_{\widehat{M}_f} & = \{ \big(h,(g',1)\big) :  (h,f(g')) \in E_{\image(f)}, h \in V_{\image_2(f)}, g' \in V_{\vG}, f(g') \notin  V_{\image_2(f)}  \} .
    \end{align*} 
\end{dfn}

\begin{dfn}
A \textbf{modified cone} $\widehat{C}_f\vG$ is a digraph $\left(V_{\widehat{C}_f\vG},E_{\widehat{C}_f\vG}\right)$ consisting of the following. The vertex set $V_{\widehat{C}_f\vG}$ is defined as $V_{\widehat{C}_f\vG} = V_{1}' \cup V_{1}'' \cup V_{C\vG}$, where 
    \begin{equation}\label{eq:ModifiedConeVertex}
    \begin{split}
        V_1' & = \{ h' \in V_{\image_2(f)}:  \text{there is } h \notin V_{\image_2(f)} \text{ such that } (h,h') \in E_{\image(f)}\}, \\
        V_{1}'' &= \{ h \in V_{\image_2(f)}:  \text{there is } h' \notin V_{\image_2(f)} \text{ such that } (h,h') \in E_{\image(f)}  \}.
    \end{split}
    \end{equation}
    For each $h \in (V_1' \cup V_1'')$, we choose a particular vertex $g=g_h \in V_{\vG}$ such that $f(g_h) = h$. We define
    \begin{equation}\label{eq:ModifiedConeEdge}
        \widehat e_h = \big((g_h,0),h \big).
    \end{equation}
    The edge set $E_{\widehat{C}_f\vG}$ is defined to be
    \begin{equation}\label{eq:ModifiedCone}
        E_{\widehat{C}_f\vG} = E_{C\vG} \coprod \tilde E_{C\vG} \coprod E_{\widehat{C}_f\vG}' \coprod E_{\widehat{C}_f\vG}'' \coprod \widehat E_{\widehat{C}_f\vG} , 
    \end{equation}
    where 
    \begin{align*}
        \tilde E_{C\vG} &= \{\big((g,0),\ast\big) : g \in V_{\vG} \},\\
        E_{\widehat{C}_f\vG}' &= \{(\ast,h') : h' \in V_1' \}, \\
        E_{\widehat{C}_f\vG}'' &= \{(h,\ast) :  h \in V_1'' \}, \\
        \widehat E_{\widehat{C}_f\vG} & = \{ \widehat{e}_h :  h \in (V_1 \cup V_1')  \} .
    \end{align*}
   \end{dfn}

    \begin{dfn}\label{DFN.digraph.cofiber}
    The \textbf{modified mapping cone} $\vec{C}(f)$ for a map $f\colon\vG\to\vec{H}$ is the digraph given by $$\vec{C}(f):=\Big(\widehat{C}_f\vG \coprod \widehat{M}_f \Big)\Big/ \sim,$$ where $V_{\widehat{C}_f\vG} \ni (g,0)\sim (g,1) \in V_{\widehat{M}_f}$ for each $g\in V_{\vec{G}}$, and $(V_1' \cup  V_1'') \ni h \sim h \in V_{\vH} \subseteq V_{\widehat{M}_f}$ for each $h \in V_1' \cup  V_1'' \subseteq V_{\vH}$.
    \end{dfn}

    \begin{rmk}\label{RMK.digraph.cofiber}
    The resulting digraph in the above definition is a modification of a ``cone on top of the mapping cylinder'' with a middle slice that serves as a copy of $\vG$. The main purpose of this construction is to ensure that Lemma~\ref{LEM.adams.3.1} and Proposition~\ref{PRP.Adams.3.5} hold. Additionally, note that, unlike in the topological setting, $\vec{C}(f)$ is \emph{not} a category-theoretic cofiber. 
    \end{rmk}

    \begin{exa}
    Let $\vG$ be the digraph $a \to b \to c$, and $\vH$ be the the digraph $0 \to 1 \to 2 \to 3 \to 0$, where vertices with the same label are identified. Consider the digraph map $f: \vG \to \vH$ defined by $f(a) = 0$ and $f(b) = f(c) = 1$. The digraphs $\widehat{M}_f$, $\widehat{C}_f\vG$, $\vec{C}(f)$ and $\widehat{M}_f \cap \widehat{C}_f\vG$ are shown in Figure~\ref{fig:cofiber}. 
    \end{exa}
    
\begin{figure}[h!]
\centering
\begin{tikzpicture}[
    myarrow/.style={-{Stealth[length=2mm]}, thick},
    bluearrow/.style={myarrow, blue},
    redarrow/.style={myarrow, red},
blackarrow/.style={myarrow, black},
purplearrow/.style={myarrow, purple},
cyanarrow/.style={myarrow, cyan},
greenarrow/.style={myarrow, green}
]
    \begin{scope}[xshift=0cm]
        \foreach \x in {0,1,2,5,8,9,10,13} {
            \fill (\x,0) circle (1.5pt);
        }
        \foreach \y in {0,1,2,4,5,6,8,9,10,12,13,14} {
            \fill (\y,1) circle (1.5pt);
        }
        \foreach \z in {5,9} {
            \fill (\z,2) circle (1.5pt);
        }
        \foreach \w in {1,9} {
            \fill (\w,-0.5) circle (1.5pt);
        }

        \node[above] at (1.3,1.1) {\tiny \textcolor{green}{$\vec G$}};
        \node[below] at (0.3,-0.2) {\tiny \textcolor{blue}{$\vec H$}};
        \node[above] at (0,1.5) { $\widehat{M}_f$};
        \node[above] at (0,1.03) {\tiny{$\text{a}$}};
        \node[above] at (0.95,1.03) {\tiny{$\text{b}$}};
        \node[above] at (1.95,1.03) {\tiny{$\text{c}$}};
        \node[below] at (0,0.03) {\tiny{$\text{0}$}};
        \node[below] at (1,0.03) {\tiny{$\text{1}$}};
        \node[below] at (2,0.03) {\tiny{$\text{2}$}};
        \node[below] at (1,-0.47) {\tiny{$\text{3}$}};
                
        \draw[greenarrow] (0,1) -- (1,1);
        \draw[greenarrow] (1,1) -- (2,1);
        \draw[bluearrow] (0,0) -- (1,0);
        \draw[bluearrow] (1,0) -- (2,0);
        \draw[bluearrow] (2,0) -- (1,-0.5);
        \draw[bluearrow] (1,-0.5) -- (0,0);
        \draw[blackarrow] (0,1) -- (0,0);
        \draw[blackarrow] (1,1) -- (1,0);
        \draw[blackarrow] (2,1) -- (1,0);
        \draw[cyanarrow] (0,1) to[bend left=15] (1,0);

        \node[above] at (5,2) {\tiny{$\ast$}};
        \node[below] at (5,0.03) {\tiny{$\text{1}$}};

        \node[above] at (3.85,1.03) {\tiny{$\text{a}$}};
        \node[above] at (4.85,1.03) {\tiny{$\text{b}$}};
        \node[above] at (6,1.03) {\tiny{$\text{c}$}};

        \node[above] at (6.3,0.7) {\tiny \textcolor{green}{$\vG$}};
        \node[above] at (3.9,1.5) {$\widehat{C}_f\vG$};
        
        \draw[greenarrow] (4,1) -- (5,1);
        \draw[greenarrow] (5,1) -- (6,1);       
        \draw[bluearrow] (6,1) to[bend left=0] (5,0);

        \draw[redarrow] (5,2) to[bend left=15] (4,1);       
        \draw[redarrow] (5,2) to[bend left=15] (5,1);
        \draw[redarrow] (5,2) to[bend left=15] (6,1);   
        \draw[redarrow] (4,1) to[bend left=15] (5,2);       
        \draw[redarrow] (5,1) to[bend left=15] (5,2);
        \draw[redarrow] (6,1) to[bend left=15] (5,2);  

        \draw[cyanarrow] (5,2) to[bend right=10] (5,0);

        \node[above] at (9,2) {\tiny{$\ast$}};
        \node[below] at (8,0.03) {\tiny{$\text{0}$}};
        \node[below] at (9,0.03) {\tiny{$\text{1}$}};
        \node[below] at (10,0.03) {\tiny{$\text{2}$}};
        \node[below] at (9,-0.47) {\tiny{$\text{3}$}};

        \node[above] at (7.85,1.03) {\tiny{$\text{a}$}};
        \node[above] at (8.85,1.03) {\tiny{$\text{b}$}};
        \node[above] at (10,1.03) {\tiny{$\text{c}$}};

        \node[above] at (7.9,1.5) {$\vec C(f)$};
        
        \draw[greenarrow] (8,1) -- (9,1);
        \draw[greenarrow] (9,1) -- (10,1);       
        \draw[bluearrow] (8,0) -- (9,0);
        \draw[bluearrow] (9,0) -- (10,0);
        \draw[bluearrow] (10,0) -- (9,-0.5);
        \draw[bluearrow] (9,-0.5) -- (8,0);
        \draw[bluearrow] (8,1) -- (8,0);
        \draw[bluearrow] (9,1) -- (9,0);
        \draw[bluearrow] (10,1) -- (9,0);
        \draw[bluearrow] (8,1) to[bend left=15] (9,0);

        \draw[redarrow] (9,2) to[bend left=15] (8,1);       
        \draw[redarrow] (9,2) to[bend left=15] (9,1);
        \draw[redarrow] (9,2) to[bend left=15] (10,1);   
        \draw[redarrow] (8,1) to[bend left=15] (9,2);       
        \draw[redarrow] (9,1) to[bend left=15] (9,2);
        \draw[redarrow] (10,1) to[bend left=15] (9,2);  
        \draw[cyanarrow] (9,2) to[bend right=10] (9,0);

        \node[above] at (13.3,1.1) {\tiny \textcolor{green}{$\vec G$}};
        \node[above] at (12,1.5) { $\widehat{M}_f \cap \widehat{C}_f\vG$};
        \node[above] at (12,1.03) {\tiny{$\text{a}$}};
        \node[above] at (12.95,1.03) {\tiny{$\text{b}$}};
        \node[above] at (13.95,1.03) {\tiny{$\text{c}$}};
        \node[below] at (13,0.03) {\tiny{$\text{1}$}};
                
        \draw[greenarrow] (12,1) -- (13,1);
        \draw[greenarrow] (13,1) -- (14,1);
        \draw[bluearrow] (14,1) to[bend left=0] (13,0);
    \end{scope}

\end{tikzpicture}
        \caption{Modified mapping cone}
        \label{fig:cofiber}
\end{figure}
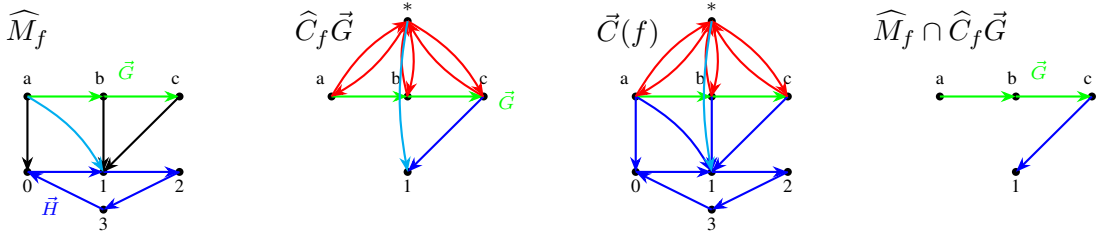

    \begin{dfn}\label{DFN.mapping.tube}
    Let $I_3$ be the digraph $0 \leftarrow 1 \rightarrow 2 \rightarrow 3$. A \textbf{mapping tube} between two digraph maps $f\colon\vec{G}\to\vec{H}$ and $g\colon\vec{G}\to\vec{H}$ is defined as 
    $$
    \overrightarrow{MT}_{f,g}= \Big((\vG \Box I_3)  \coprod \vH \Big) \Big/ \sim,
    $$
    where $(x,0) \sim f(x)$ and $(x,3) \sim g(x)$ for $x \in V_{\vG}$. 
    \end{dfn}

    \begin{rmk}
        Let \( f \amalg g: \vG \coprod \vG \to \vH \) be the digraph map that sends \( x_l \) to \( f(x_l) \) and \( x_r \) to \( g(x_r) \), where \( x_l \) and \( x_r \) denote the vertices in the left and right copies of \( \vG \) in \( \vG \coprod \vG \), respectively, corresponding to a vertex \( x \) in \( \vG \). The mapping tube $\overrightarrow{MT}_{f,g}$ can be decomposed as 
        \begin{equation}\label{eq:rmk:MappingTube}
            \overrightarrow{MT}_{f,g} = \vec{M}_{f \amalg g} \cup (\vG \Box I^+).
        \end{equation}
        The intersection of the two subdigraphs is $\vec{M}_{f \amalg g} \cap (\vG \Box I^+) = \vG \coprod \vG$.
    \end{rmk}

    \begin{exa}
        Let $\vG$ be the digraph $a \to b \to c$, and $\vH$ be the the digraph $0 \to 1 \to 2 \to 3 \to 0$, where vertices with the same label are identified. Consider the digraph map $f,g: \vG \to \vH$ defined by $f(a) = 0$, $f(b) = f(c) = 1 =g(a)$ and $g(b)=g(c) = 2$. The digraph $\overrightarrow{MT}_{f,g}$ is shown in Figure~\ref{fig:mtfg}.
    \end{exa}

    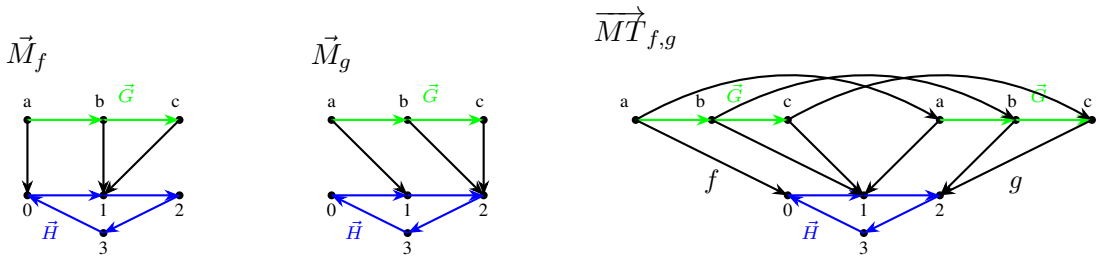
\begin{figure}[h!]
        \centering
\centering
\begin{tikzpicture}[
    myarrow/.style={-{Stealth[length=2mm]}, thick},
    bluearrow/.style={myarrow, blue},
    redarrow/.style={myarrow, red},
blackarrow/.style={myarrow, black},
purplearrow/.style={myarrow, purple},
cyanarrow/.style={myarrow, cyan},
greenarrow/.style={myarrow, green}
]
    \begin{scope}[xshift=0cm]
        \foreach \x in {0,1,2,4,5,6,10,11,12} {
            \fill (\x,0) circle (1.5pt);
        }
        \foreach \y in {0,1,2,4,5,6,8,9,10,12,13,14} {
            \fill (\y,1) circle (1.5pt);
        }
        \foreach \w in {1,5,11} {
            \fill (\w,-0.5) circle (1.5pt);
        }

        \node[above] at (1.3,1.1) {\tiny \textcolor{green}{$\vec G$}};
        \node[below] at (0.3,-0.2) {\tiny \textcolor{blue}{$\vec H$}};
        \node[above] at (0,1.5) { $\vec{M}_f$};
        \node[above] at (0,1.03) {\tiny{$\text{a}$}};
        \node[above] at (0.95,1.03) {\tiny{$\text{b}$}};
        \node[above] at (1.95,1.03) {\tiny{$\text{c}$}};
        \node[below] at (0,0.03) {\tiny{$\text{0}$}};
        \node[below] at (1,0.03) {\tiny{$\text{1}$}};
        \node[below] at (2,0.03) {\tiny{$\text{2}$}};
        \node[below] at (1,-0.47) {\tiny{$\text{3}$}};
                
        \draw[greenarrow] (0,1) -- (1,1);
        \draw[greenarrow] (1,1) -- (2,1);
        \draw[bluearrow] (0,0) -- (1,0);
        \draw[bluearrow] (1,0) -- (2,0);
        \draw[bluearrow] (2,0) -- (1,-0.5);
        \draw[bluearrow] (1,-0.5) -- (0,0);
        \draw[blackarrow] (0,1) -- (0,0);
        \draw[blackarrow] (1,1) -- (1,0);
        \draw[blackarrow] (2,1) -- (1,0);

        \node[above] at (5.3,1.1) {\tiny \textcolor{green}{$\vec G$}};
        \node[below] at (4.3,-0.2) {\tiny \textcolor{blue}{$\vec H$}};
        \node[above] at (4,1.5) { $\vec{M}_g$};
        \node[above] at (4,1.03) {\tiny{$\text{a}$}};
        \node[above] at (4.95,1.03) {\tiny{$\text{b}$}};
        \node[above] at (5.95,1.03) {\tiny{$\text{c}$}};
        \node[below] at (4,0.03) {\tiny{$\text{0}$}};
        \node[below] at (5,0.03) {\tiny{$\text{1}$}};
        \node[below] at (6,0.03) {\tiny{$\text{2}$}};
        \node[below] at (5,-0.47) {\tiny{$\text{3}$}};
                
        \draw[greenarrow] (4,1) -- (5,1);
        \draw[greenarrow] (5,1) -- (6,1);
        \draw[bluearrow] (4,0) -- (5,0);
        \draw[bluearrow] (5,0) -- (6,0);
        \draw[bluearrow] (6,0) -- (5,-0.5);
        \draw[bluearrow] (5,-0.5) -- (4,0);
        \draw[blackarrow] (4,1) -- (5,0);
        \draw[blackarrow] (5,1) -- (6,0);
        \draw[blackarrow] (6,1) -- (6,0);

        \node[below] at (10,0.03) {\tiny{$\text{0}$}};
        \node[below] at (11,0.03) {\tiny{$\text{1}$}};
        \node[below] at (12,0.03) {\tiny{$\text{2}$}};
        \node[below] at (11,-0.47) {\tiny{$\text{3}$}};

        \node[above] at (7.85,1.03) {\tiny{$\text{a}$}};
        \node[above] at (8.85,1.03) {\tiny{$\text{b}$}};
        \node[above] at (10,1.03) {\tiny{$\text{c}$}};
        \node[below] at (10.3,-0.2) {\tiny \textcolor{blue}{$\vec H$}};        
        \node[above] at (9.3,1.1) {\tiny \textcolor{green}{$\vec G$}};
        \node[above] at (13.3,1.1) {\tiny \textcolor{green}{$\vec G$}};
        \node[above] at (12,1.03) {\tiny{$\text{a}$}};
        \node[above] at (12.95,1.03) {\tiny{$\text{b}$}};
        \node[above] at (13.95,1.03) {\tiny{$\text{c}$}};
        
        \node[above] at (8,1.8) {$\overrightarrow{MT}_{f,g}$};
        \node[above] at (9,-0.1) {\small{$f$}};
        \node[above] at (13,-0.1) {\small{$g$}};
        
        \draw[greenarrow] (8,1) -- (9,1);
        \draw[greenarrow] (9,1) -- (10,1);       
        \draw[greenarrow] (12,1) -- (13,1);
        \draw[greenarrow] (13,1) -- (14,1);
        
        \draw[bluearrow] (10,0) -- (11,0);
        \draw[bluearrow] (11,0) -- (12,0);
        \draw[bluearrow] (12,0) -- (11,-0.5);
        \draw[bluearrow] (11,-0.5) -- (10,0);
        
        \draw[blackarrow] (8,1) -- (10,0);
        \draw[blackarrow] (9,1) -- (11,0);
        \draw[blackarrow] (10,1) -- (11,0);
        \draw[blackarrow] (12,1) -- (11,0);
        \draw[blackarrow] (13,1) -- (12,0);
        \draw[blackarrow] (14,1) -- (12,0);
        \draw[blackarrow] (8,1) to[bend left=30] (12,1);
        \draw[blackarrow] (9,1) to[bend left=30] (13,1);
        \draw[blackarrow] (10,1) to[bend left=30] (14,1);
        
    \end{scope}

\end{tikzpicture}        \caption{A mapping tube between $f$ and $g$}
        \label{fig:mtfg}
    \end{figure}
    
\subsection{Homotopy for digraphs}
Two digraph maps $f,g\colon\vec{G}\to\vec{H}$ are \textbf{homotopic}, denoted $f\simeq g$, if there exists an $n\geq 1$ and a digraph map $F\colon\vec{G}\Box I_n\to \vec{H}$, for some line digraph $I_n\in\mathcal{I}_n$, such that $F\mid_{\vec{G}\times\{0\}}=f$ and $F\mid_{\vec{G}\times\{n\}}=g$. For every vertex $i\in V_{I_n}$, $F\mid_{\vec{G}\times \{i\}}$ must be a digraph map from $\vec{G}$ to $\vec{H}$. Thus, if two digraph maps, $f$ and $g$, are homotopic, then there must be a sequence of digraph maps, $\{f_j\}_{j=0}^n$, where $f_0=f$, $f_n=g$, and $f_j=F\mid_{\vec{G}\Box\{j\}}$ for $0<j<n$. We denote by $[f]$ the set of all digraph maps which are homotopic to $f$.

    \begin{exa}\label{ex:CofiberHpt}
        Let $f: \vG \to \vH$ be a digraph map which naturally induces a digraph map $\tilde f: \vG \to \vec{C}(f)$. Let $c$ be the constant map $\vG \to \vec{C}(f)$ mapping to the top point $\ast$. 
        Define a homotopy $F: \vG \Box I_2 \to \vec{C}(f)$ by 
        \begin{align*}
            F(g,0) &= f(g) \in \vH \hookrightarrow \vec{C}(f), \\
            F(g,1) &= g \in \vG  \hookrightarrow \widehat{C}_f\vG  \hookrightarrow \vec{C}(f), \\
            F(g,2) &= c(g) = \ast \in \widehat{C}_f\vG \hookrightarrow \vec{C}(f).
        \end{align*}
        This homotopy $F$ also shows that $\tilde f$, $c$ and the inclusion map $\vG \hookrightarrow \vec{C}(f)$ are homotopic.
    \end{exa}
    
Two digraphs are said to be \textbf{homotopically equivalent} (or of the same {\bf homotopy type}) if there exist two digraph maps, $g\colon\vec{G}\to\vec{H}$ and $h\colon\vec{H}\to\vec{G}$, such that $h\circ g\simeq \text{id}_{\vec{G}}$ and $g\circ h\simeq\text{id}_{\vec{H}}$. Such $g$ and $h$ are called {\bf homotopy equivalences}.  We will denote the class of all digraphs which are homotopically equivalent to $\vec{G}$ by $[\vec{G}]$ and every element of this set is said to be of the homotopy type of $\vec{G}$.

    \begin{exa}\label{EXA.mapping.cylinder.crushing} Let $f\colon\vG\ra\vH$ be a digraph map. The modified mapping cylinder $\widehat{M}_f$ and the digraph $\vH$ are homotopically equivalent. This can be shown by taking a homotopy $F\colon\widehat{M}_f\Box I^+\ra \widehat{M}_f$ defined by $F(-,0)=\text{id}_{\widehat{M}_f}$,  $F((g,0),1)=f(g)$ for all $g\in V_{\vG}$ and $F(h,1)=h$ for all $h\in V_{\vH}$.
    \end{exa}

A digraph $\vec{G}$ is said to be \textbf{contractible} if there exists a homotopy between $\text{id}_{\vec{G}}$ and a constant digraph map. We will need the following lemma.


    

    \begin{lem}\label{lem:ModifiedCone}
        A modified cone $\widehat{C}_f\vG$ is contractible.
    \end{lem}
    \begin{proof}
        Let $c:\widehat{C}_f\vG \to \widehat{C}_f\vG$ be the constant digraph map with value $\ast \in \widehat{C}_f\vG$, and let $I_2$ be the digraph $0  \leftarrow 1 \leftarrow  2 $. 
         Define a homotopy $F:\widehat{C}_f\vG  \Box I_2  \to \widehat{C}_f\vG$ by requiring $F(-,0) = \id_{\widehat{C}_f\vG}$, $F(-,2) = c$,
        $$F(x,1) = \begin{cases}
            (g_x,0) & \text{ if } x \in V_1 \cup V_1', \\
            x & \text{ if } x \notin V_1 \cup V_1'.
        \end{cases}$$
        See \eqref{eq:ModifiedConeVertex}, \eqref{eq:ModifiedConeEdge} and \eqref{eq:ModifiedCone} for the notations.  
        It is straightforward to verify that $F$ is a homotopy between $\id_{\widehat{C}_f\vG}$ and $c$. This proves the lemma.
    \end{proof}

Digraph maps $f,g\colon \vec{G}\to \vec{H}$ are said to be \textbf{weakly homotopic}, denoted $f\simeq_w g$, if for every $\vec{K}\in\mathcal{D}_0$ and every digraph map $h\colon \vec{K}\to \vec{G}$, compositions $f\circ h$ and $g\circ h$ are homotopic. Let $[\vec{G},\vec{H}]_w$ denote the set of weak homotopy classes of digraph maps from $\vec{G}$ to $\vec{H}$. Note that, when $\vec{G}$ is a finite digraph,  $[\vec{G},\vec{H}]_w=[\vec{G},\vec{H}]$.

The \textbf{homotopy category of directed graphs}, denoted $\text{Ho}\mathcal{D}$, is a category in which the objects are directed graphs and the morphisms are homotopy equivalence classes of digraph maps. The {\bf homotopy category of finite directed graphs} $\text{Ho}\mathcal{D}_0$ and the {\bf weak homotopy category for directed graphs} $\text{wHo}\mathcal{D}$ are defined in the same manner. The category $\text{Ho}\mathcal{D}$ is a homotopy category in a category-theoretic sense. It is isomorphic to the localization $\mathcal{W}^{-1}\mathcal{D}$ of the category $\mathcal{D}$ with respect to the collection $\mathcal{W}$ of homotopy equivalences.

\begin{rmk}\label{EXA.HEP.does.not.hold}
    A key foundation in the classical homotopy theory is the homotopy extension property. It is well-known that if $X$ is a CW complex and $A$ a subcomplex, then the pair $(X,A)$ has the homotopy extension property. Nevertheless, such a property does \emph{not} hold in the category of digraphs. This makes the development of a digraph homotopy theory nontrivial. Here, a pair $(\vG,\vec{X})$ of a digraph and its subdigraph is said to have the {\em homotopy extension property} if given a map $f\colon \vG\to \vH$ and a homotopy $F: \vec{X} \Box I_n\to \vH$ such that $F(-,0)=f|_{\vec{X}}$, there exists $\widehat{F}\colon \vG \Box I_n \to \vH$ such that $\widehat{F}|_{\vec{X} \Box I_n}=F$ and $\widehat{F}(-,0) = f$.
    
    As a concrete example that the homotopy extension property fails, let us consider the digraph $ \vG =\vH = a \to b \to c \to a$ where the vertex with the same label are identified. See Figure~\ref{fig:triangle}. Let $\vec{X} = a \to b$, $f = \id_{\vG}$ and $F:\vec{X} \Box I^+ \to \vH$ be the homotopy given by $F(x,0) = x$ and $F(x,1) = b$ for $x=a,b$. It is clear that $F(-,0) = f|_{\vec{X}}$, and we claim that $F$ cannot be extended to $\widehat{F}:\vG \Box I^+ \to \vH$ such that $\widehat{F}(-,0) = f = \id_{\vG}$.

    If there exists such an extension $\widehat{F}$, then $\widehat{F}(x,0) = x$, for $x=a,b,c$, and $\widehat{F}(a,1) = \widehat{F}(b,1) = b$. There are three possibilities of $\widehat{F}(c,1)$: $a,b$ and $c$. (i) If $\widehat{F}(c,1) =a$, then the edge $(b,1) \to (c,1)$ is sent to $b \to a$ which is not an edge in $\vH$. (ii) If $\widehat{F}(c,1) =b$, then the edge $(c,0) \to (c,1)$ is sent to $c \to b$ which is not an edge in $\vH$.  (iii) If $\widehat{F}(c,1) =c$, then the edge $(c,1) \to (a,1)$ is sent to $c \to b$ which is not an edge in $\vH$. Therefore, we conclude that there is no such an extension $\widehat{F}.$
\end{rmk}

    \begin{figure}[h!]
        \centering
\centering
\begin{tikzpicture}[
    myarrow/.style={-{Stealth[length=2mm]}, thick},
    bluearrow/.style={myarrow, blue},
    redarrow/.style={myarrow, red},
blackarrow/.style={myarrow, black},
purplearrow/.style={myarrow, purple},
cyanarrow/.style={myarrow, cyan},
greenarrow/.style={myarrow, green}
]
    \begin{scope}[xshift=0cm]
        \foreach \x in {1,5} {
            \fill (\x,0) circle (1.5pt);
        }
        \foreach \y in {0,2,4,6} {
            \fill (\y,1) circle (1.5pt);
        }
        \foreach \z in {5} {
            \fill (\z,2) circle (1.5pt);
        }
        \foreach \w in {4,6} {
            \fill (\w,3) circle (1.5pt);
        }

        \node[above] at (0.3,0.1) {\tiny \textcolor{green}{$\vec{X}$}};
        \node[above] at (-0.2,1.5) { $\vG=\vH$};
        \node[above] at (0,1.03) {\tiny{$a$}};
        \node[above] at (0.95,0.03) {\tiny{$b$}};
        \node[above] at (1.95,1.03) {\tiny{$c$}};
                
        \draw[greenarrow] (0,1) -- (1,0);
        \draw[bluearrow] (1,0) -- (2,1);
        \draw[bluearrow] (2,1) -- (0,1);

        \node[above] at (7,1.5) { $\vG \Box I^+$};
        \node[below] at (3.7,1) {\tiny{$(a,0)$}};
        \node[below] at (5,0) {\tiny{$(b,0)$}};
        \node[below] at (6.3,1) {\tiny{$(c,0)$}}; 
        \node[above] at (3.7,3) {\tiny{$(a,1)$}};
        \node[above] at (5,2.2) {\tiny{$(b,1)$}};
        \node[above] at (6.3,3) {\tiny{$(c,1)$}};        
        \draw[bluearrow] (4,1) -- (5,0);
        \draw[bluearrow] (5,0) -- (6,1);
        \draw[bluearrow] (6,1) -- (4,1);
        \draw[bluearrow] (4,3) -- (5,2);
        \draw[bluearrow] (5,2) -- (6,3);
        \draw[bluearrow] (6,3) -- (4,3);
        \draw[blackarrow] (5,0) -- (5,2);
        \draw[blackarrow] (6,1) -- (6,3);
        \draw[blackarrow] (4,1) -- (4,3);        
    \end{scope}

\end{tikzpicture}        \caption{$\vG$ and $\vG \Box I^+$}
        \label{fig:triangle}
\end{figure}

\section{Brown functors and their properties}\label{SEC.section.3}
In this section, we define digraph Brown functors for finite digraphs and study their properties. Our approach is inspired by Adams' work \cite{Ad} on classical topological spaces.

\begin{nta} Consider a diagram $B\srl{f}\ra A\srl{g}\leftarrow C$ in the category of abelian groups \textbf{Ab}. We will use the notation $B\times_A C$ to denote the subset of $B\times C$ defined by $\{(b,c) : f(b)=g(c)\}$.
\end{nta}

\begin{dfn}\label{DFN.brown.functor} A \textbf{(digraph) Brown functor} is a functor $\BF \colon \hD_0^{\text{op}}\ra \textbf{Ab}$ satisfying the following axioms \begin{itemize}
    \item[(1)] \textbf{Triviality Axiom.} The functor $\BF$ sends a singleton to the trivial group.
    \item[(2)] \textbf{Additivity Axiom.} The functor $\BF$ sends coproduct to product. i.e. $\BF(\coprod_{\al\in \La}\vG_\al)=\prod_{\al\in \La} \BF(\vG_\al)$ for any family of digraphs $\{\vG_\al\}_{\al\in \La}$.
    \item[(3)] \textbf{Mayer-Vietoris Axiom.} For any digraphs $\vG_1,\vG_2\in \hD_0$, the map $\BF(\vG_1\cup \vG_2)\ra \BF(\vG_1)\times_{\BF(\vG_1\cap\vG_2)} \BF(\vG_2)$ induced by the inclusions is a surjection. 
\end{itemize}
\end{dfn}


Although our main interest lies in the study of finite digraphs, to establish the Brown representability theorem, we need to extend a Brown functor to the category $\cD$ of arbitrary digraphs, including both finite and infinite ones. More explicitly, given a Brown functor $\BF$, it can be extended to a functor on $\cD$ as follows. 

Let $\vec G \in \cD$. Consider $\{ \vec G_\alpha\}_{\alpha \in \Lambda}$, the set of all the finite subdigraphs of $\vec G$. Together with inclusions, it forms a directed set. We extend $\BF$ to a functor $\EBF$ by taking the inverse limit over this directed set:
\begin{equation}\label{eq:ExtendedBF-0}
\EBF(\vG):=\underset{\alpha \in \Lambda}{\varprojlim} \BF(\vG_\alpha).
\end{equation}
Explicitly, we have the identification 
\begin{equation}\label{eq:ExtendedBF}
    \EBF(\vG) = \Big\{ (x_\alpha)_{\alpha \in \Lambda}  \in \prod_{\alpha \in \Lambda} \BF(\vG_\alpha): \f i_{\alpha \beta}^\ast (x_\beta) = x_\alpha, \forall \vG_\alpha \subseteq  \vG_\beta \fsub  \vG \Big\},
\end{equation}
where the notation $\f i_{\alpha\beta}$ means the inclusion map $\f i_{\alpha \beta}: \vG_\alpha \hookrightarrow  \vG_\beta$, and $\vH \fsub \vG$ means that $\vH$ is a finite subdigraph of $\vG$.

Given a morphism $f:\vG \to \vH$ in $\cD$, we have the induced homomorphism $f^\ast: \EBF(\vH) \to \EBF(\vG)$ given by 
\begin{equation}\label{eq:EBFInducedHomo}
    f^\ast\big((y_\beta)_{\beta}\big) = \big(f^\ast(y_{f(\alpha)}) \big)_{\alpha},
\end{equation}
where $f(\alpha)$ is the index for the subdigraph $\vH_{f(\alpha)} = f(\vG_\alpha)$. 

The following properties of $\EBF$ can be easily verified, as in \cite{Ad}.

\begin{lem} The assignment $\EBF\colon \hD^{\text{op}}\ra \textbf{Ab}$, $\vG\mapsto \underset{\alpha}{\varprojlim} \BF(\vG_\alpha)$ is a functor which restricts to the functor $\BF\colon \hD_0^{\text{op}}\ra \textbf{Ab}$. Also an assignment $\EBF\colon \whD^{\text{op}}\ra \textbf{Ab}$ defined similarly is a functor that restricts to the functor $\BF\colon \hD_0^{\text{op}}\ra \textbf{Ab}$. 
\end{lem}

\begin{lem}\label{LEM.adams.3.4}
Let $\vG\in\cD$ and $\{\vG_\al\}_{\al\in\La'}$ be any directed set of (not necessarily finite) subdigraphs of $\vG$ such that $\bigcup_{\al \in \Lambda'} \vG_\al = \vG$. Then there is a canonical isomorphism \beqs
\Te\colon \EBF(\vG)&\ra\underset{\alpha \in \La'}{\varprojlim} \EBF(\vG_\alpha)\\
 x&\mapsto (\f i_\al^*x)_{\al\in\La'}
\eeqs where $\f i_\al:\vG_\al\embed \vG$ are the inclusion maps.
\end{lem}

\begin{lem}\label{lem:EBF-additivity} The functor $\EBF$ satisfies the additivity axiom.
\end{lem}

The next lemma is parallel to \cite[Lemma~3.1]{Ad}. However, a modified mapping cone $\vec{C}(f)$ of a digraph map $f$ is \emph{not} a category-theoretic cofiber, which introduces a technical difference. Therefore, we include a proof of the lemma. 

\begin{lem}\label{LEM.adams.3.1}
Let $f\colon\vec{G}\ra \vec{H}$ be a map of finite digraphs and $\BF$ a Brown functor. The sequence 
$$
\begin{tikzcd}
    \BF(\vec{G}) & \ar[l,"f^\ast"'] \BF(\vec{H}) & \ar[l,"i^\ast"'] \BF(\vec{C}(f))
\end{tikzcd}
$$ 
induced by the sequence $\vec{G}\srl{f}\ra \vec{H}\srl{i}\ra \vec{C}(f)$ of digraphs is exact at $\BF(\vec{H})$. Here $i:\vH \hookrightarrow \vec{C}(f)$ is the natural embedding map. 
\end{lem}
\begin{proof}
First note that the composition $i \circ f$ is homotopic to a constant map $c:\vG \to \vec{C}(f)$ by Example~\ref{ex:CofiberHpt}. Since the constant map $c$ can be obtained by the composition 
$$
\begin{tikzcd}
    \vG \ar[r] & \ast \ar[r] & \vec{C}(f)
\end{tikzcd}
$$
and $\BF(\ast) = 0$, it follows that $f^\ast \circ i^\ast = c^\ast =0$, i.e.\ $\image(i^\ast)\subseteq \ker(f^\ast)$. 

To show $\ker(f^\ast) \subseteq \image(i^\ast)$, recall that $\vec{C}(f)$ can be decomposed as $\vec{C}(f) = \widehat{C}_f \vG \cup \widehat{M}_f$. By a homotopy similar to the one in the proof of Lemma~\ref{lem:ModifiedCone}, one can show that $\widehat{C}_f\vG \cap \widehat{M}_f$ is homotopically equivalent to $\vG$. See Figure~\ref{fig:cofiber}. By the Mayer-Vietoris property of $\BF$, the map
\begin{equation}\label{eq:LEM.adams.3.1-1}
    \BF(\vec{C}(f)) \to \BF(\widehat{C}_f \vG) \times_{\BF(\vG)} \BF(\widehat{M}_f) 
\end{equation}
induced by inclusions is surjective. Furthermore, since $\widehat{C}_f \vG$ is contractible (by Lemma~\ref{lem:ModifiedCone}), and $\vH \hookrightarrow \widehat{M}_f$ is a homotopy equivalence (by Example~\ref{EXA.mapping.cylinder.crushing}), we have $ \BF(\widehat{C}_f \vG) =0$ and $\BF(\widehat{M}_f) \cong \BF(\vH)$. Since the inclusion map $\vG \hookrightarrow \widehat{M}_f$ is homotopic to the composition $\vG \xrightarrow{f} \vH \hookrightarrow \widehat{M}_f$ (via a homotopy similar to the one in Example~\ref{ex:CofiberHpt}), the surjectivity of the map \eqref{eq:LEM.adams.3.1-1} is equivalent to the surjectivity of the following map
\begin{equation}\label{eq:LEM.adams.3.1-2}
    \BF(\vec{C}(f)) \to 0 \times_{\BF(\vG)} \BF(\vH), \quad y \mapsto (0,i^\ast y),
\end{equation}
where the fiber product is over the diagram $0 \xrightarrow{\; 0\; } \BF(\vG) \xleftarrow{f^\ast} \BF(\vH)$. Now given any $x \in \ker(f^\ast)$, we have $(0,x) \in 0 \times_{\BF(\vG)} \BF(\vH)$, and it follows from the surjectivity of \eqref{eq:LEM.adams.3.1-2} that there exists $y \in  \BF(\vec{C}(f))$ such that $x = i^\ast(y) \in \image(i^\ast)$. This completes the proof.
\end{proof}

The following lemma can be obtained by applying Lemma~\ref{LEM.adams.3.1} to the following sequence of digraphs: 
\begin{equation}\label{eq:LEM.adams.3.2}
\vG\coprod \vH\srl{f}\ra \vG\cup \vH \srl{g}\ra \vec{C}(f)\srl{h}\ra \vec{C}(g). 
\end{equation}

\begin{lem}\label{LEM.adams.3.2} 
Let $\vG$ and $\vH$ be finite digraphs. There is an exact sequence \beq\label{EQN.Adams.Lemma.3.2.exact.seq} \BF(\vG)\times \BF(\vH)\isom \BF(\vG\coprod \vH)\srl{f^*}\leftarrow \BF(\vG\cup \vH) \srl{g^*}\leftarrow \BF(\vec{C}(f))\srl{h^*}\leftarrow \BF(\vec{C}(g)),\eeq which is natural in $\vG$ and $\vH$.
\end{lem}

The next proposition is parallel to \cite[Proposition 3.5]{Ad}, which is crucial. Although the structure of the proof is similar to Adams', there are several technical differences. The key idea in the proof of \cite[Proposition 3.5]{Ad} is the use of the suspension of the finite intersection of the given CW complexes, and our main difficulty is that the suspension does not work well in our setting. To fix this problem, we consider modified mapping cones and a technical finite digraph $\vec{S}$ that plays the role of a suspension in our proof.

\begin{prp}\label{PRP.Adams.3.5} Let $\vG,\vH\in \cD$ and $\vG\cap\vH\in \cD_0$. Then the map 
\begin{equation}\label{eq:PRP.Adams.3.5}
\EBF(\vG\cup\vH)\ra \EBF(\vG)\times_{\EBF(\vG\cap\vH)}\EBF(\vH)    
\end{equation}
induced by the inclusion maps $\vG\embed \vG\cup\vH$ and $\vH\embed \vG\cup\vH$ is onto.
\end{prp}
\begin{proof}
To prove Proposition~\ref{PRP.Adams.3.5}, recall the concept of generalized inverse limits in \cite[Section~2]{Ad}. 

Let 
\begin{align*}
    \{\vG_\al\}_{\al\in\La} & = \Big\{ \vG_\al \fsub \vG : \vG_\al \supseteq \vG \cap \vH \Big\}, \\
    \{\vH_\be\}_{\be\in\Ga} & = \Big\{ \vH_\be \fsub \vH : \vH_\be \supseteq \vG \cap \vH \Big\}.
\end{align*}
Take any $(x,y)\in \EBF(\vG)\times_{\BF(\vG\cap\vH)}\EBF(\vH)$. This means that for every $\al \in \La$ we have an element $x_\al\in \BF(\vG_\al)$ as a restriction of $x\in \EBF(\vG)$, and for every $\be \in \Ga$ we have an element $y_\be\in \BF(\vH_\be)$ as a restriction of $y\in \EBF(\vH)$ such that $(x_\al,y_\be)\in \BF(\vG_\al)\times_{\BF(\vG\cap\vH)}\BF(\vH_\be)$. Notice here that $\vG_\al\cap\vH_\be\supseteq \vG\cap\vH$. 

Let $\psi_{\al\be}:\BF(\vG_\al\cup\vH_\be) \to \BF(\vG_\al)\times_{\BF(\vG\cap\vH)}\BF(\vH_\be)$ be the map induced by the inclusions, and let 
\begin{equation}\label{eq:PRP.Adams.3.5-0}
    H_{\al,\be}:=\psi_{\al\be}^{-1}(x_\al,y_\be) \subseteq \BF(\vG_\al\cup\vH_\be).
\end{equation}
By the Mayer-Vietoris axiom, the sets $H_{\al,\be}$ are nonempty. We consider the subcategory $\cC$ of $\textbf{Set}$ whose objects are the sets $H_{\al,\be}$ for $\al \in \La$ and $\be \in \Ga$. For any objects  $H_{\al,\be},H_{\ga,\de}$ in $ \cC$,  the hom-set $\cC(H_{\ga,\de},H_{\al,\be})$ is empty if $\vG_\al\cup\vH_\be \not\subseteq \vG_\ga\cup\vH_\de$, and consists of a single map, the restriction map, if $\vG_\al\cup\vH_\be \subseteq \vG_\ga\cup\vH_\de$. 

To show the surjectivity of the map \eqref{eq:PRP.Adams.3.5}, it suffices to show that the limit $\varprojlim \cC$ is nonempty. By \cite[Corollary~2.8]{Ad}, we need to verify that 
\begin{itemize}
    \item[(i)] every morphism in $\cC$ is onto, and 
    \item[(ii)] the objects in $\overline{\cC}$ fall into countably many equivalence classes. 
\end{itemize}
See \cite[Section~2]{Ad} for the definition of $\overline{\cC}$.

To prove (i) and (ii), we need two key observations.

First of all, we claim that the group $\BF(\vec{C}(f_{\al\be}))$ acts on $H_{\al,\be}$ transitively, for any $\al \in \La$, $\be \in \Ga$. Explicitly, let $f:\vG \coprod \vH \to \vG \cup \vH$ and $g:\vG \cup \vH \to \vec{C}(f)$ be the maps defined in \eqref{eq:LEM.adams.3.2}, and let $f_{\al\be}:\vG_\al \coprod \vH_\be \to \vG_\al \cup \vH_\be$ and $g_{\al\be}:\vG_\al \cup \vH_\be \to \vec{C}(f_{\al\be})$ be the parallel maps. Here, we choose $\vec{C}(f_{\al\be})$ so that they are naturally embedded into $\vec{C}(f)$. The action $\BF(\vec{C}(f_{\al\be})) \times H_{\al,\be} \to H_{\al,\be}$ is defined by 
\begin{equation}\label{eq:PRP.Adams.3.5-1}
    u \cdot w =  g_{\al\be}^\ast(u) + w, \qquad \forall u\in \BF(\vec{C}(f_{\al\be})), w \in H_{\al,\be}.
\end{equation}
 Since $f_{\al\be}^\ast g_{\al\be}^\ast =0$, we have $\psi_{\al\be}(g_{\al\be}^\ast(u)) = 0$, and thus Equation~\eqref{eq:PRP.Adams.3.5-1} indeed defines an action. For transitivity, given any $w_1, w_2\in H_{\al,\be}$, since $w_1-w_2 \in \ker(f_{\al\be}^\ast)$, if follows from Lemma~\ref{LEM.adams.3.2} that there is $u \in \BF(\vec{C}(f_{\al\be}))$ such that $w_1-w_2=g_{\al\be}^*(u)$.

Secondly, let $\vec{S}$ be the part of $\vec{C}(f)$ over $\vG \cap \vH$, i.e.\ the induced subdigraph of $\vec{C}(f)$ whose vertex set is $V_{\vec{S}} = V_{\vG \cap \vH} \cup ( V_{\vG \cap \vH} \coprod  V_{\vG \cap \vH}) \times \{0\} \cup \{ \ast\}$. 
The digraph $\vec{S}$ is embedded into $\vec{C}(f_{\al\be})$ and $\vec{C}(f)$. We claim that the inclusion map 
\begin{equation}\label{eq:PRP.Adams.3.5-intersection}
    \vec{S} \hookrightarrow \vec{C}(f_{\al\be})
\end{equation}
is a homotopy equivalence for any $\al \in \La$, $\be \in \Ga$. Note that $\vG \cap \vH = \vG_\al \cap \vH_\be$ for any $\al \in \La$, $\be \in \Ga$. Thus, it suffices to show that $\vec{S} \hookrightarrow \vec{C}(f)$ is a homotopy equivalence for any digraphs $\vG$ and $\vH$. Let $j:\vec{S} \hookrightarrow \vec{C}(f)$ be the inclusion map, and $r:\vec{C}(f) \to \vec{S}$ be the digraph map  
\begin{equation}\label{eq:PRP.Adams.3.5-2}
    r(x) = \begin{cases}
    x' & \text{ if } x=j(x') \in j(\vec{S}), \\
    \ast & \text{ if } x \notin j(\vec{S}).
\end{cases}
\end{equation}
It is clear that $r \circ j = \id_{\vec{S}}$, and it remains to show that $j \circ r \simeq \id_{\vec{C}(f)}$.  Let $c:\vec{C}(f) \to \vec{C}(f)$ be the constant digraph map with value $\ast \in \vec{C}(f)$, and let $I_2$ be the digraph $0  \leftarrow 1 \leftarrow  2 $.  Define a homotopy $F:\vec{C}(f)  \Box I_2  \to \vec{C}(f)$ by requiring $F(-,0) = \id_{\vec{C}(f)}$, $F(-,2) = j \circ r$,
\begin{equation}\label{eq:PRP.Adams.3.5-homotopy}
    F(x,1) = \begin{cases}
            (x,0) \in V_{\widehat{C}_f(\vG \coprod \vH)}, & \text{ if } x \in (V_{\vG \cup \vH} \setminus V_{\vG \cap \vH}) \subseteq V_{\vG \cup \vH}  \subseteq V_{\vec{C}(f)}, \\
            x  \in V_{\vec{C}(f)}, & \text{ if } x \notin V_{\vG \cup \vH} \setminus V_{\vG \cap \vH}.
            \end{cases}
\end{equation}
This proves the second claim. See Figure~\ref{fig:suspension}. 

\begin{figure}[h!]
\centering
\begin{tikzpicture}[
    myarrow/.style={-{Stealth[length=2mm]}, thick},
    bluearrow/.style={myarrow, blue},
    redarrow/.style={myarrow, red},
blackarrow/.style={myarrow, black},
purplearrow/.style={myarrow, purple},
cyanarrow/.style={myarrow, cyan},
greenarrow/.style={myarrow, green}
]
    \begin{scope}[xshift=0cm]
        \foreach \a in {2,3,4,9,10} {
            \fill (\a,0) circle (1.5pt);
        }
        \foreach \b in {0,1,5,7,8,12} {
            \fill (\b,1) circle (1.5pt);
        }
        \foreach \c in {2,3,4,10,11} {
            \fill (\c,2) circle (1.5pt);
        }
        \foreach \d in {1,2,3,9,10,11} {
            \fill (\d,3) circle (1.5pt);
        }
        \foreach \e in {0,1,2,6,7,8,9,12} {
            \fill (\e,4) circle (1.5pt);
        }
        \foreach \f in {9,10,11} {
            \fill (\f,5) circle (1.5pt);
        }

        \node[above] at (1.5,4.1) {\tiny \textcolor{green}{$\vec G$}};
        \node[above] at (2.5,3.1) {\tiny \textcolor{blue}{$\vec H$}};
        \node[above] at (0,4) {\tiny{$\text{a}$}};
        \node[above] at (1,4) {\tiny{$\text{b}$}};
        \node[above] at (2,4) {\tiny{$\text{c}$}};
        \node[below] at (1,3) {\tiny{$\text{b}$}};
        \node[below] at (2,3) {\tiny{$\text{c}$}};
        \node[below] at (3,3) {\tiny{$\text{d}$}};
                
        \draw[greenarrow] (0,4) -- (1,4);
        \draw[greenarrow] (1,4) -- (2,4);
        \draw[bluearrow] (1,3) -- (2,3);
        \draw[bluearrow] (2,3) -- (3,3);
        \draw[double equal sign distance] (1,3.8) -- (1,3.2);
        \draw[double equal sign distance] (2,3.8) -- (2,3.2);

        \node[above] at (12.2,3.8) {\tiny{$\ast$}};
        \node[above] at (6,4) {\tiny{$\text{a}$}};
        \node[above] at (7,4) {\tiny{$\text{b}$}};
        \node[above] at (8,4) {\tiny{$\text{c}$}};
        \node[above] at (9,4) {\tiny{$\text{d}$}};
        \node[above] at (9,5) {\tiny{$\text{a}$}};
        \node[above] at (10,5) {\tiny{$\text{b}$}};
        \node[above] at (11,5) {\tiny{$\text{c}$}};             \node[below] at (9,3) {\tiny{$\text{b}$}};
        \node[below] at (10,3) {\tiny{$\text{c}$}};
        \node[below] at (11,3) {\tiny{$\text{d}$}};
        \node[above] at (6,4.5) {\tiny $\vec{C}(f) = \image(F(-,0))$};
        \node[above] at (6.5,3.5) {\tiny $\vG \cup \vH$};
        \node[above] at (9.7,5) {\tiny\textcolor{green}{$\vec G$}};
        \node[below] at (9.7,3) {\tiny  \textcolor{blue}{$\vec H$}};

        \draw[greenarrow] (9,5) -- (10,5);
        \draw[greenarrow] (10,5) -- (11,5);       
        \draw[greenarrow] (9,5) to[bend left=3] (12,4);
        \draw[greenarrow] (12,4) to[bend left=3] (9,5); 
        \draw[greenarrow] (10,5) to[bend left=3] (12,4);
        \draw[greenarrow] (12,4) to[bend left=3] (10,5); 
        \draw[greenarrow] (11,5) to[bend left=4] (12,4);
        \draw[greenarrow] (12,4) to[bend left=4] (11,5); 
        
        \draw[bluearrow] (9,3) -- (10,3);
        \draw[bluearrow] (10,3) -- (11,3);
        \draw[bluearrow] (9,3) to[bend left=3] (12,4);
        \draw[bluearrow] (12,4) to[bend left=3] (9,3); 
        \draw[bluearrow] (10,3) to[bend left=3] (12,4);
        \draw[bluearrow] (12,4) to[bend left=3] (10,3); 
        \draw[bluearrow] (11,3) to[bend left=4] (12,4);
        \draw[bluearrow] (12,4) to[bend left=4] (11,3); 

        \draw[blackarrow] (6,4) -- (7,4);
        \draw[blackarrow] (7,4) -- (8,4);
        \draw[blackarrow] (8,4) -- (9,4);
        \draw[blackarrow] (9,5) -- (6,4);
        \draw[blackarrow] (10,5) -- (7,4);
        \draw[blackarrow] (11,5) -- (8,4);
        \draw[blackarrow] (9,3) -- (7,4);
        \draw[blackarrow] (10,3) -- (8,4);
        \draw[blackarrow] (11,3) -- (9,4);        
        
        \draw[redarrow] (9,5) -- (7,4);       
        \draw[redarrow] (8,4) -- (11,3);
        \draw[redarrow] (12,4) to[bend right=15] (7,4);   
        \draw[redarrow] (8,4) to[bend right=15] (12,4);

        \node[above] at (5.2,0.8) {\tiny{$\ast$}};
        \node[above] at (0,1) {\tiny{$\text{b}$}};
        \node[above] at (1,1) {\tiny{$\text{c}$}};
        \node[above] at (2,2) {\tiny{$\text{a}$}};
        \node[above] at (3,2) {\tiny{$\text{b}$}};
        \node[above] at (4,2) {\tiny{$\text{c}$}};             \node[below] at (2,0) {\tiny{$\text{b}$}};
        \node[below] at (3,0) {\tiny{$\text{c}$}};
        \node[below] at (4,0) {\tiny{$\text{d}$}};        \node[above] at (-0.5,1.5) {\tiny $ \image(F(-,1))$};
        \node[above] at (-0.5,0.5) {\tiny $\vG \cap \vH$};
        \node[above] at (2.7,2) {\tiny\textcolor{green}{$\vec G$}};
        \node[below] at (2.7,0) {\tiny  \textcolor{blue}{$\vec H$}};
        
        \draw[greenarrow] (2,2) -- (3,2);
        \draw[greenarrow] (3,2) -- (4,2);       
        \draw[greenarrow] (2,2) to[bend left=3] (5,1);
        \draw[greenarrow] (5,1) to[bend left=3] (2,2); 
        \draw[greenarrow] (3,2) to[bend left=3] (5,1);
        \draw[greenarrow] (5,1) to[bend left=3] (3,2); 
        \draw[greenarrow] (4,2) to[bend left=4] (5,1);
        \draw[greenarrow] (5,1) to[bend left=4] (4,2); 
        
        \draw[bluearrow] (2,0) -- (3,0);
        \draw[bluearrow] (3,0) -- (4,0);
        \draw[bluearrow] (2,0) to[bend left=3] (5,1);
        \draw[bluearrow] (5,1) to[bend left=3] (2,0); 
        \draw[bluearrow] (3,0) to[bend left=3] (5,1);
        \draw[bluearrow] (5,1) to[bend left=3] (3,0); 
        \draw[bluearrow] (4,0) to[bend left=4] (5,1);
        \draw[bluearrow] (5,1) to[bend left=4] (4,0); 

        \draw[blackarrow] (0,1) -- (1,1);
        \draw[blackarrow] (3,2) -- (0,1);
        \draw[blackarrow] (4,2) -- (1,1);
        \draw[blackarrow] (2,0) -- (0,1);
        \draw[blackarrow] (3,0) -- (1,1);
        
        \draw[redarrow] (2,2) -- (0,1);       
        \draw[redarrow] (1,1) -- (4,0);
        \draw[redarrow] (5,1) to[bend right=15] (0,1);   
        \draw[redarrow] (1,1) to[bend right=15] (5,1);  

        \node[above] at (12.2,0.8) {\tiny{$\ast$}};
        \node[above] at (7,1) {\tiny{$\text{b}$}};
        \node[above] at (8,1) {\tiny{$\text{c}$}};
        \node[above] at (10,2) {\tiny{$\text{b}$}};
        \node[above] at (11,2) {\tiny{$\text{c}$}};             \node[below] at (9,0) {\tiny{$\text{b}$}};
        \node[below] at (10,0) {\tiny{$\text{c}$}};
        \node[above] at (6.5,1.5) {\tiny $\vec{S} = \image(F(-,2))$};
        \node[above] at (6.5,0.5) {\tiny $\vG \cap \vH$};
        \node[above] at (10.5,2.1) {\tiny\textcolor{green}{$\vG \cap \vH$}};
        \node[below] at (9.5,-0.1) {\tiny  \textcolor{blue}{$\vG \cap \vH$}};        
        
        \draw[greenarrow] (10,2) -- (11,2);       
        \draw[greenarrow] (10,2) to[bend left=3] (12,1);
        \draw[greenarrow] (12,1) to[bend left=3] (10,2); 
        \draw[greenarrow] (11,2) to[bend left=4] (12,1);
        \draw[greenarrow] (12,1) to[bend left=4] (11,2); 
        
        \draw[bluearrow] (9,0) -- (10,0);
        \draw[bluearrow] (9,0) to[bend left=3] (12,1);
        \draw[bluearrow] (12,1) to[bend left=3] (9,0); 
        \draw[bluearrow] (10,0) to[bend left=3] (12,1);
        \draw[bluearrow] (12,1) to[bend left=3] (10,0);

        \draw[blackarrow] (7,1) -- (8,1);
        \draw[blackarrow] (10,2) -- (7,1);
        \draw[blackarrow] (11,2) -- (8,1);
        \draw[blackarrow] (9,0) -- (7,1);
        \draw[blackarrow] (10,0) -- (8,1);
        
        \draw[redarrow] (12,1) to[bend right=15] (7,1);   
        \draw[redarrow] (8,1) to[bend right=15] (12,1);      
    \end{scope}

\end{tikzpicture}
        \caption{Homotopy between $\id_{\vec{C}(f)}$ and $j \circ r$}
        \label{fig:suspension}
\end{figure}
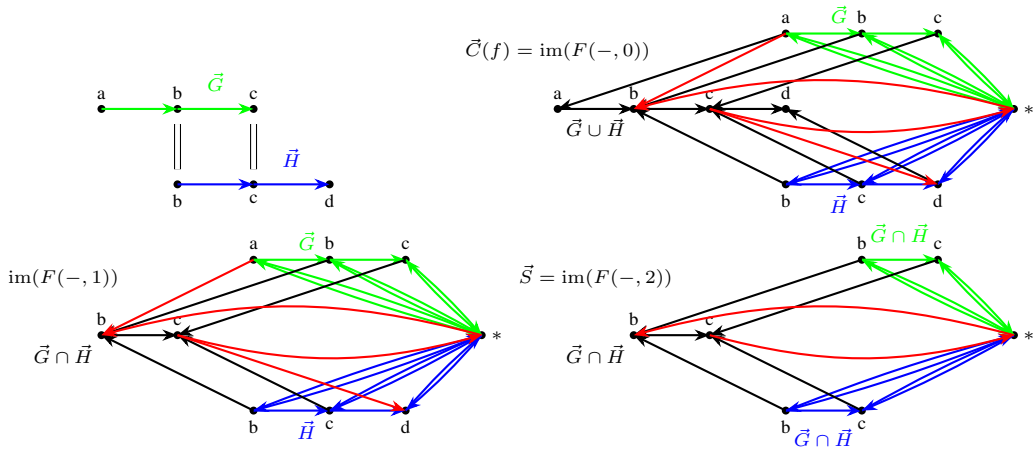

\emph{Proof of (i).}
It follows from these two claims that $\BF(\vec{S})$ acts on $H_{\al,\be}$ transitively for any $\al \in \La$, $\be \in \Ga$:
\begin{equation}\label{eq:PRP.Adams.3.5-3}
    u \cdot w =  g_{\al\be}^\ast r_{\al\be}^\ast(u) + w, \qquad \forall u\in \BF(\vec{S}), w \in H_{\al,\be},
\end{equation}
where $r_{\al\be}:\vec{C}(f_{\al\be}) \to \vec{S}$ is defined by the formula \eqref{eq:PRP.Adams.3.5-2}. 
Since both $g_{\al\be}$ and $r_{\al\be}$ commute with the embeddings $\vG_\al \cup \vH_\be \hookrightarrow \vG_\ga \cup \vH_\de$ and $\vec{C}(f_{\al\be}) \hookrightarrow \vec{C}(f_{\ga\de})$, the restriction map $\imath^*\colon H_{\ga,\de}\ra H_{\al,\be}$ commutes with the action:
$$
\imath^\ast(u \cdot w) = u \cdot \imath^\ast(w), \qquad \forall u\in \BF(\vec{S}), w \in H_{\ga,\de}.
$$
Let $w_0 \in H_{\ga,\de}$. By transitivity, given any element  $w \in H_{\al,\be}$, there exists $u\in \BF(\vec{S})$ such that $w = u \cdot \imath^\ast(w_0)$, and thus 
$$
\imath^\ast( u \cdot w_0) = w.
$$
This proves that the restriction map $\imath^*\colon H_{\ga,\de}\ra H_{\al,\be}$ is surjective.

\emph{Proof of (ii).}
To prove (ii), we need a description of the hom-set $\overline{\cC}(H_{\ga,\de},H_{\al,\be})$. Let $h:\vec{C}(f) \to \vec{C}(g)$ be the map defined in \eqref{eq:LEM.adams.3.2}, and let $h_{\al\be}:\vec{C}(f_{\al\be}) \to \vec{C}(g_{\al\be})$ be the parallel map. Combining $h_{\al\be}$ with the homotopy equivalence $j_{\al\be}:\vec{S} \hookrightarrow \vec{C}(f_{\al\be})$, we have the map
$$
 j_{\al\be}^\ast h_{\al\be}^\ast: \BF(\vec{C}(g_{\al\be})) \to \BF(\vec{S}) \cong \BF(\vec{C}(f_{\al\be})).
$$
We claim that the hom-set $\overline{\cC}(H_{\ga,\de},H_{\al,\be})$ is nonempty if and only if $ j_{\ga\de}^\ast h_{\ga\de}^\ast\big(\BF(\vec{C}(g_{\ga\de}))\big) \subseteq  j_{\al\be}^\ast h_{\al\be}^\ast\big(\BF(\vec{C}(g_{\al\be}))\big)$.

Recall that $\f k \in \overline{\cC}(H_{\ga,\de},H_{\al,\be})$ if and only if there exists $H_{\te,\phi}$ such that the diagram
\begin{equation}\label{eq:PRP.Adams.3.5-4}
\begin{tikzcd}
 & \ar[ld,"\imath_{\ga\de}^\ast"']  H_{\te,\phi} \ar[rd,"\imath_{\al\be}^\ast"] & \\
 H_{\ga,\de} \ar[rr,"\f k"']   & &  H_{\al,\be}
\end{tikzcd}
\end{equation}
commutes, where $\imath_{\al\be}: \vG_\al \cup \vH_\be \hookrightarrow \vG_\theta \cup \vH_\phi$ and $\imath_{\ga\de}: \vG_\ga \cup \vH_\de \hookrightarrow \vG_\theta \cup \vH_\phi$ are the inclusion maps. Consider the commutative diagram 
\begin{equation}\label{eq:PRP.Adams.3.5-5}
    \begin{tikzcd}
       \BF(\vG_{\al}) \times \BF(\vH_{\be}) & \ar[l,"f_{\al\be}^\ast"'] \BF(\vG_{\al} \cup \vH_{\be}) & \ar[l,"g_{\al\be}^\ast r_{\al\be}^\ast"'] \BF(\vec{S}) & \ar[l,"j_{\al\be}^\ast h_{\al\be}^\ast"']\BF(\vec{C}(g_{\al\be})) \\
       \BF(\vG_{\te}) \times \BF(\vH_{\phi}) \ar[u] \ar[d] & \ar[l,"f_{\te\phi}^\ast"'] \BF(\vG_{\te} \cup \vH_{\phi}) \ar[u,"\imath_{\al\be}^\ast"] \ar[d,"\imath_{\ga\de}^\ast"'] & \ar[l,"g_{\te\phi}^\ast r_{\te\phi}^\ast"'] \BF(\vec{S}) \ar[u,equal] \ar[d,equal] & \ar[l,"j_{\te\phi}^\ast h_{\te\phi}^\ast"']\BF(\vec{C}(g_{\te\phi})) \ar[u] \ar[d] \\
       \BF(\vG_{\ga}) \times \BF(\vH_{\de}) & \ar[l,"f_{\ga\de}^\ast"'] \BF(\vG_{\ga} \cup \vH_{\de}) & \ar[l,"g_{\ga\de}^\ast r_{\ga\de}^\ast"'] \BF(\vec{S}) & \ar[l,"j_{\ga\de}^\ast h_{\ga\de}^\ast"']\BF(\vec{C}(g_{\ga\de})), 
    \end{tikzcd}
\end{equation}
where the retractions $r_{\al\be}: \vec{C}(f_{\al\be}) \to \vec{S}$ are defined by the formula \eqref{eq:PRP.Adams.3.5-2}, $j_{\al\be}: \vec{S} \hookrightarrow \vec{C}(f_{\al\be})$ are the inclusion maps, and the vertical arrows are induced by the inclusion maps.

Now assume that there exists $\f k \in \overline{\cC}(H_{\ga,\de},H_{\al,\be})$. If $ j_{\ga\de}^\ast h_{\ga\de}^\ast\big(\BF(\vec{C}(g_{\ga\de}))\big) \not\subseteq  j_{\al\be}^\ast h_{\al\be}^\ast\big(\BF(\vec{C}(g_{\al\be}))\big)$, then, by Lemma~\ref{LEM.adams.3.2}, there exists $u$ such that $g_{\ga\de}^\ast r_{\ga\de}^\ast(u) =0$ but $g_{\al\be}^\ast r_{\al\be}^\ast(u) \neq 0$. This implies that 
$$
\f k(0) =\f k(g_{\ga\de}^\ast r_{\ga\de}^\ast(u)) =\f k(\imath_{\ga\de}^\ast g_{\te\phi}^\ast r_{\te\phi}^\ast(u)) =\imath_{\al\be}^\ast g_{\te\phi}^\ast r_{\te\phi}^\ast(u))  = g_{\al\be}^\ast r_{\al\be}^\ast(u)  \neq  0,
$$
which is a contradiction since 
$$
\f k(0) = \f k (\imath_{\ga\de}^\ast(0)) = \imath_{\al\be}^\ast(0) =0.
$$
Therefore, we conclude that $ j_{\ga\de}^\ast h_{\ga\de}^\ast\big(\BF(\vec{C}(g_{\ga\de}))\big) \subseteq  j_{\al\be}^\ast h_{\al\be}^\ast\big(\BF(\vec{C}(g_{\al\be}))\big)$.

Conversely, assume that $ j_{\ga\de}^\ast h_{\ga\de}^\ast\big(\BF(\vec{C}(g_{\ga\de}))\big) \subseteq  j_{\al\be}^\ast h_{\al\be}^\ast\big(\BF(\vec{C}(g_{\al\be}))\big)$. By Lemma~\ref{LEM.adams.3.2}, it is equivalent to that $\ker(g_{\ga\de}^\ast r_{\ga\de}^\ast) \subseteq \ker(g_{\al\be}^\ast r_{\al\be}^\ast)$. Let $w_0$ be a fixed element in $H_{\theta,\phi}$. By the transitivity of the action \eqref{eq:PRP.Adams.3.5-3}, each element in $H_{\ga,\de}$ is of the form $u \cdot \iota_{\ga\de}^\ast(w_0)$, $u \in \BF(\vec{S})$. Define $\f k:H_{\ga,\de} \to H_{\al,\be}$ by
$$
\f k( u \cdot \iota_{\ga\de}^\ast(w_0)) = u \cdot \iota_{\al\be}^\ast(w_0).
$$
The assumption that $\ker(g_{\ga\de}^\ast r_{\ga\de}^\ast) \subseteq \ker(g_{\al\be}^\ast r_{\al\be}^\ast)$ guarantees the well-definedness of $\f k$, and the commutativity of the diagram \eqref{eq:PRP.Adams.3.5-4} follows from the fact that the restrictions commute with the action \eqref{eq:PRP.Adams.3.5-3}. This completes the proof of our claim.

The claim implies that two objects $H_{\al,\be}$ and $H_{\ga,\de}$ are equivalent in $\overline{\cC}$ if and only if the equality $ j_{\ga\de}^\ast h_{\ga\de}^\ast\big(\BF(\vec{C}(g_{\ga\de}))\big) = j_{\al\be}^\ast h_{\al\be}^\ast\big(\BF(\vec{C}(g_{\al\be}))\big)$ holds in $\BF(\vec{S})$. Note that since $\vG \cap \vH$ is finite, the digraph $\vec{S}$ is also finite. Thus, both $j_{\ga\de}^\ast h_{\ga\de}^\ast$ and $j_{\al\be}^\ast h_{\al\be}^\ast$ are induced by maps of \emph{finite} digraphs. Since there are only countably many homotopy classes of maps from finite digraphs to the finite digraph $\vec{S}$, the possible images of the induced maps are also countably many. Therefore, there are only countably many equivalence classes of objects in $\overline{\cC}$. This completes the proof of (ii) and hence the proof of Proposition~\ref{PRP.Adams.3.5}.
\end{proof}

\section{Brown's Method for Directed Graphs}\label{SEC.section.4}
In this section, we construct a classifying object representing a Brown functor on finite digraphs. The arguments in this section are parallel to those in Brown~\cite{B62} and Adams~\cite{Ad}. Notice that our main contribution is the construction of mapping tubes in Definition~\ref{DFN.mapping.tube}, which is necessary for Lemma~\ref{induct}. 

Throughout this section we will use the notation $\Nat(F,G)$ to denote the set of all natural transformations from a functor $F$ to a functor $G$.

Similar to \cite[Section~4]{Ad}, we need another description of the extended Brown functor $\EBF$ via the Yoneda lemma: 
Let $\BF \colon \text{Ho}\mathcal{D}^{\text{op}}_0\to\textbf{Set}$ be a functor. To extend $\BF$ to infinite digraphs, consider $\vec{Y}\in\cD$ and its associated functor $[-,\vec{Y}]:\text{Ho}\mathcal{D}^{\text{op}}_0\to\textbf{Set}$ which sends a finite digraph $\vG$ to the set $[\vG,\vec{Y}]$ of homotopy classes of digraph maps from $\vG$ to $\vec{Y}$. 
Define 
\begin{equation}\label{eq:ExtendedFunctorByYoneda}
    \bar{\BF}(\vec{Y}):=\Nat([-,\vec{Y}],\BF(-)).
\end{equation}
If $\vec{Y}$ is finite, then by the Yoneda lemma, we have the isomorphism
$$
  \BF(\vec{Y}) \xrightarrow{\cong} \bar\BF(\vec{Y}), \quad  y \mapsto T_y,
$$
where $T_y$ is the natural transformation given by 
\begin{equation}\label{eq:NatTransHclass}
    T_{y,\vG}: [\vG,\vec{Y}] \to \BF(\vG), \quad [f] \mapsto f^\ast(y), 
\end{equation}
for any finite digraph $\vG$.

The following lemma can be easily verified, as in \cite[Lemma~4.1]{Ad} (or \cite[Lemma~3.3]{B62}).

    \begin{lem}
    \label{LEM.Brown.Lemma.3.3}
    There is an isomorphism between sets:
    $$
    \phi\colon \bar{\BF}(\vec{Y})\to \EBF(\vec{Y}) = \varprojlim_\al \BF(\vec{Y}_\alpha), \quad T \mapsto \Big(T_{\vec{Y}_\alpha}([\f i_\al]) \Big)_\al,
    $$
    where the limit is taken over all the finite subdigraphs $\vec{Y}_\alpha$ of $\vec{Y}$, $T \in \Nat([-,\vec{Y}],\BF(-))$, and $\f i_\al: \vec{Y}_\al \hookrightarrow \vec{Y}$ are the inclusion maps. 
    \end{lem}

    Let $\BF$ be a Brown functor which induces $\EBF:\text{wHo}\cD^{\text{op}}\ra \textbf{Set}$ by Equation~\eqref{eq:ExtendedBF-0}. By the Yoneda lemma, for each digraph $\vec{Y} \in \cD$, we have 
    $$
    \EBF(\vec{Y}) \cong \Nat([-,\vec{Y}]_w, \EBF(-)).
    $$
    Combining with Lemma~\ref{LEM.Brown.Lemma.3.3}, we have
    \begin{equation}
        \Nat([-,\vec{Y}]_w,\EBF(-))\cong \EBF(\vec{Y})\cong \Nat([-,\vec{Y}],\BF(-)),
    \end{equation}
    where $[-,\vec{Y}]_w$ and $\EBF(-)$ are functors from $\text{wHo}\cD^{\text{op}}$ to $\textbf{Set}$, and $[-,\vec{Y}]$ and $\BF(-)$ are functors from $\text{Ho}\cD_0^{\text{op}}$ to $\textbf{Set}$.

    Recall that we constructed mapping tubes $ \overrightarrow{MT}_{f,g}$ in Definitions~\ref{DFN.mapping.tube}.  We will need the following properties of mapping tubes.

    \begin{lem} \label{MTl}
    Let $f,g\colon \vec{G}\to \vec{H}$ be digraph maps, and $i\colon \vec{H}\to \overrightarrow{MT}_{f,g}$ be the natural embedding. Then $i\circ f\simeq i\circ g$.
    \end{lem}
    \begin{proof}
        Let $I_3$ be the digraph $0 \leftarrow 1 \rightarrow 2  \rightarrow 3$, and $F: \vG \Box I_3 \to \overrightarrow{MT}_{f,g}$ be the composition of digraph maps 
        $$
        \begin{tikzcd}
         \vG \Box I_3 \ar[r,"\id"] &  \vG \Box I_3 \ar[r,hook] & (\vG \Box I_3)  \coprod \vH \ar[r,two heads] &  \overrightarrow{MT}_{f,g},
        \end{tikzcd}
        $$
        where the rightmost map is the quotient map. 
        It is clear that $F$ defines a homotopy between $i \circ f$ and $i \circ g$, which proves the lemma.
    \end{proof}

Let $i\colon \vec{H}\to \overrightarrow{MT}_{f,g}$ be the natural embedding, and $j: \vG \to \overrightarrow{MT}_{f,g}$ be the embedding $j(g) = (g,1)$, $g \in V_{\vG}$. Note that there is another embedding $j': \vG \to \overrightarrow{MT}_{f,g}$, $j'(g) = (g,2)$, which is homotopic to $j$.

    \begin{lem} \label{MT2}
    Let $\vG$ be a finite digraph, and $f,g\colon \vec{G}\to \vec{H}$ be digraph maps. Suppose that $\BF$ is a Brown functor. The map 
    \begin{equation}\label{eq:MT2}
        \EBF(\overrightarrow{MT}_{f,g}) \to \EBF(\vG) \times_{\EBF(\vG \coprod \vG)} \EBF(\vH), \quad x \mapsto (j^\ast x, i^\ast x),
    \end{equation}
    is surjective. Here the fiber product is over the diagonal map $\EBF(\vG) \to \EBF(\vG) \oplus \EBF(\vG) = \EBF(\vG \coprod \vG)$ and the map $\EBF(\vH) \to  \EBF(\vG) \oplus \EBF(\vG), x \mapsto (f^\ast x, g^\ast x)$.
    \end{lem}
    \begin{proof}
        Recall that we have the decomposition \eqref{eq:rmk:MappingTube} of a mapping tube $\overrightarrow{MT}_{f,g}$. Thus, by Proposition~\ref{PRP.Adams.3.5}, the map 
        $$
        \EBF(\overrightarrow{MT}_{f,g}) \to \EBF(\vG\Box I^+) \times_{\EBF(\vG \coprod \vG)} \EBF(\vec{M}_{f \amalg g}),
        $$
        induced by the inclusions, is onto. Since $\vG$ is homotopy equivalent to $\vG \Box I^+$, and $\vH$ is homotopy equivalent to $\vec{M}_{f \amalg g}$, we conclude that the map \eqref{eq:MT2} is also onto. This proves the lemma.
    \end{proof}
    
    
The following lemma is similar to \cite[Lemma~4.2]{Ad}. 

    \begin{lem}\label{induct}
    Let $\BF$ be a Brown functor. 
    Suppose $\vec{Y}_n$ is a digraph equipped with an element $y_n\in\EBF(\vec{Y}_n)$. Then there exist a digraph $\vec{Y}_{n+1}$, an embedding $i\colon \vec{Y}_n\to \vec{Y}_{n+1}$ and an element $y_{n+1}\in\EBF(\vec{Y}_{n+1})$ such that the following conditions hold:
    \begin{itemize}
        \item[(i)]
         $i^\ast(y_{n+1}) = y_n$;
        \item[(ii)]
        for any pair of digraph maps $f,g\colon \vec{K}\to \vec{Y}_n$, where $\vec{K}$ is a finite digraph, if $f^\ast y_n=g^\ast y_n$, then $i\circ f\simeq i\circ g: \vec{K} \to \vec{Y}_{n+1}$. 
    \end{itemize}
    \end{lem}

    To prove the lemma, one just needs to replace the construction $Y_{n+1} = Y_n \cup \bigcup_{\al \in A} (I \times K_\al)/(I \times \text{pt})$ in the proof of \cite[Lemma~4.2]{Ad} by 
    \begin{equation*}
         \vec{Y}_{n+1}= \bigcup_{\alpha\in A} \big(\overrightarrow{MT}_{f_{\alpha},g_{\alpha}} \big).
    \end{equation*}

    The following lemma is similar to \cite[Proposition~4.4]{Ad}.

    \begin{lem}\label{LEM.representability.on.finite.digraph}
    Given a digraph $\vec{Y}_0$ and an element $y_0\in\EBF(\vec{Y}_0)$, there exist a digraph $\vec{Y}$, an embedding $i\colon \vec{Y}_0\to\vec{Y}$ and an element $y\in\EBF(\vec{Y})$ such that the following hold:
    \begin{itemize}
        \item[(i)] $i^\ast(y) = y_0$;
        \item [(ii)] the map $T_{y,\vG} \colon [\vG,\vec{Y}]\to \BF(\vG)$ is a bijection for each finite digraph $\vG$, where $T_y$ is the natural transformation defined by \eqref{eq:NatTransHclass}.
    \end{itemize}
    \end{lem}

    As in \cite[Proposition~4.4]{Ad}, we choose a representative $\vec{K}_\al$ from each homotopy type of finite digraphs and form a countable set $\{\vec{K}_\alpha\}_{\alpha\in A}$ by collecting all the chosen representatives. Consider 
    \begin{equation*}\label{eq:Y1}
         \vec{Y}_1 = \vec{Y}_0 \amalg \coprod_{\alpha \in A} \Big(\coprod_{\lambda \in \BF(\vec{K}_\alpha)} \vec{K}_\alpha \Big),        
    \end{equation*}
    where the disjoint union $\coprod_{\lambda \in \BF(\vec{K}_\alpha)} \vec{K}_\alpha$ contains $n_\al = |\BF(\vec{K}_\alpha)|$ copies of $\vec{K}_\alpha$. By Lemma~\ref{lem:EBF-additivity}, we have $\EBF(\vec{Y}_1) = \EBF(\vec{Y}_0)\times\prod_{\alpha \in A} \left (\prod_{\lambda \in \BF(\vec{K}_\alpha)}\EBF(\vec{K}_\alpha) \right )$. Thus, there exists $y_1 \in \EBF(\vec{Y}_1)$ that restricts to $y_0 \in \BF(\vec{Y}_0)$ and to $\lambda \in \BF(\vec{K}_\alpha)$ at the $(\al, \lambda)$-component. By Lemma~\ref{induct}, we have 
    $$
    \vec{Y}_1\sqsubset \vec{Y}_2\sqsubset \vec{Y}_3\sqsubset\cdots\sqsubset \vec{Y}_n\sqsubset\cdots$$
    together with elements $y_n\in \EBF(\vec{Y}_n)$, $n=1,2,\cdots$, such that $y_{n+1}\in\EBF(\vec{Y}_{n+1})$ restricts to $y_n\in\EBF(\vec{Y}_n)$ for each $n$. Now, we define
    \begin{equation*}
        \vec{Y}=\bigcup_n Y_n.
    \end{equation*}
    By Lemma~\ref{LEM.adams.3.4}, there exists an element $y\in\EBF(\vec{Y})$ such that $i^\ast_n(y)=y_n$ for each $n$. It can be verified that this pair $(Y,y)$ satifies the properties in Lemma~\ref{LEM.representability.on.finite.digraph}.

    By Lemma~\ref{LEM.representability.on.finite.digraph}~(ii), we have our main theorem. 
    \begin{thm}\label{THM.main.theorem}
    Let $\BF \colon \hD_0^{\text{op}}\to \textbf{Ab}$ be a Brown functor. Then there exist a digraph $\vec{Y}$ and a natural isomorphism $T\colon [-,\vec{Y}]\to \BF(-)$.
    \end{thm}

\section{Path cohomology groups as Brown functors}\label{sec:PathCoh}

The main purpose of this section is to provide a nontrivial example of Brown functor. More precisely, we will show that the first path cohomology of digraphs is a Brown functor (Proposition~\ref{prop:H1Brown}).

\subsection{Path cohomology of digraphs}
In this subsection, we briefly recall the construction of path cohomology groups. See, for example, \cite{GLMY12,GLMY15} for more details. 

Let $\vec{G}= (V,E)$ be a digraph. An {\bf elementary (allowed) $p$-path} in $\vec{G}$ is a sequence $i_0,\cdots, i_p$ of $(p+1)$ vertices of $\vec G$, denoted by $e_{i_0 \cdots i_p}$, such that $(i_j,i_{j+1}) \in E$ for each $j$.  We denote by $\mathcal{A}_p(\vec G)$ the free abelian group generated by all the elementary allowed $p$-path in $\vec{G}$, which is naturally a subgroup of the free abelian group $\Lambda_p(V)$ generated by all the $(p+1)$-tuples in $V$. Consider the boundary operator $\partial:\Lambda_p(V) \to \Lambda_{p-1}(V)$ defined by 
$$ \partial e_{i_0\cdots i_{p}}=\sum\limits^{p}_{k=0}(-1)^k e_{i_0\cdots \widehat{i_k}\cdots i_{p}},
$$
which clearly satisfies the equation $\partial \circ \partial =0$. Note that the subgroups $I_p(V)$ of $\Lambda_p(V)$, generated by $e_{i_0 \cdots i_p}$ with $i_k = i_{k+1}$ for some $k$, form a subcomplex of $(\Lambda_\bullet (V), \partial)$, and thus we have the quotient complex $(\cR_\bullet (V), \partial)$ with $\cR_p(V) = \Lambda_p(V)/I_p(V)$. The group $\cA_p(\vG)$ can be naturally embedded into $\cR_p(V)$, but the pair $(\mathcal{A}_\bullet(\vec G), \partial)$ does \emph{not} form a chain complex since $\partial\big((\mathcal{A}_p(\vec G)\big) \not\subseteq \mathcal{A}_{p-1}(\vec G)$. To obtain a chain complex, we consider the following subgroup of $\mathcal{A}_p(\vec G)$:
$$
\Omega_p(\vec G) = \{ v \in \mathcal{A}_p(\vec G) : \partial v \in \mathcal{A}_{p-1}(\vec G) \}.
$$
It is easy to see that $(\Omega_\bullet(\vec G),\partial)$ is indeed a chain complex, and its homology is referred to as the {\bf (path) homology} of $\vec G$. 

\begin{rmk}\label{rmk:Omega_01}
Since $\mathcal{A}_p(\vec{G})$ is a free abelian group, its subgroup $\Omega_p(\vec{G})$ is also a free abelian group. However, the structure of a basis for $\Omega_p(\vec{G})$ is not immediately clear. In general, we can only describe bases for $\Omega_0(\vec{G})$ and $\Omega_1(\vec{G})$.  

For $p = 0$, the space $\Omega_0(\vec G)$ of $0$-chains is defined as $\mathcal{A}_0(\vec G)$, which is the free abelian group generated by the vertices of the given digraph $\vec G$.  

For $p = 1$, since for any edge $e = (i_1, i_2) \in E$, its boundary satisfies $\partial e = i_2 - i_1 \in \mathcal{A}_0(\vec G)$, the space $\Omega_1(\vec G)$ is the free abelian group with $E$ as its basis.
\end{rmk}

Now we consider the dual complex of $(\Omega_\bullet(\vec G),\partial)$ which will be denoted by $(\Omega^\bullet(\vec G), d)$. Explicitly, 
$$
\Omega^p(\vec G) = \Hom_{\Z}(\Omega_p(\vec G), \Z)
$$
and $d$ is the dual operator of $\partial$. The $p$-th cohomology group of the cochain complex $(\Omega^\bullet(\vec G), d)$ is referred to as the {\bf $p$-th (path) cohomology group} (with coefficients in $\Z$) of $\vec G$, and will be denoted by $H^p(\vec G) = H^p(\vec G,\Z)$.

\begin{exa}
    If $\vG =\ast$ is a singleton, then $\Omega^0(\ast) \cong \Z$ and $\Omega^p(\ast) =0$ for $p \neq 0$. Thus, $H^0(\ast) \cong \Z$ and $H^p(\ast) =0$ for $p \neq 0$.
\end{exa}

For the case $p=0$, we have the following

\begin{prp}\label{prop:H0}
Let $\vec G = (V,E)$ be a digraph with $k$ connected components $\vec G_1, \cdots, \vec G_k$. The zeroth cohomology group of $\vec G$ can be identified with the set of maps
$\{\vec G_1,\cdots,\vec G_k\}\rightarrow\mathbb{Z}.$
\end{prp}

As an immediate consequence of the proof of \cite[Theorem~3.3]{GLMY14}, we have 

\begin{prp}
Let $\vec G, \vec H$ be two digraphs. If $f\simeq g: \vec G \rightarrow \vec H$ are homotopic digraph maps, then they induce the same map in cohomology:
$$
f^{\ast}=g^{\ast}:H^{p}(\vec H)\rightarrow H^p(\vec G), \qquad \forall p\geq 0.
$$
In particular, the $p$-th cohomology induces a functor $H^p \colon \hD_0^{\text{op}}\ra \textbf{Ab}$.
\end{prp}

\subsection{Brown functor properties of path cohomology}

We first consider the zeroth cohomology $H^0$. Although $H^0$ is not a Brown functor ($H^0(\ast) \neq 0$), it is still representable. In fact, by Proposition~\ref{prop:H0}, we have 

\begin{prp}\label{prop:H0Brown}
    As functors $\hD_0^{\text{op}}\to \textbf{Ab}$, there is a natural isomorphism between $[-,\Z]$ and $H^0(-)$. Here, the notation $\Z$ refers to the digraph whose set of vertices is the set of integers and whose set of edges is empty. 
\end{prp}

To prove that $H^1 \colon \hD_0^{\text{op}}\ra \textbf{Ab}$ is a Brown functor, we need to verify that it satisfies the additivity axiom and Mayer-Vietoris axiom in Definition~\ref{DFN.brown.functor}. The proof of the additivity axiom is straightforward, so we will focus on the Mayer-Vietoris axiom.

Let $\vec G_1$ and $\vec G_2$ be two subdigraphs of a digraph $\vec G$ so that $\vec G = \vec G_1 \cup \vec G_2$. The inclusion maps
\begin{equation}\label{eq:GraphInclusion}
	\begin{tikzcd}
		\vec G_1\cap \vec G_2 \arrow[r,hook, "\f j_2"] \arrow[d,hook, "\f j_1"]
		& \vec G_2 \arrow[d,hook, "\f i_2"] \\
		\vec G_1 \arrow[r,hook, "\f i_1"]
		& \vec G
	\end{tikzcd}
\end{equation}
induce commutative diagrams 
\begin{equation}\label{eq:InducedByInclusion}
\begin{tikzcd}
		\Omega^p(\vec{G}_1\cap \vec{G}_2) \arrow[r,leftarrow, "\f j_2^*"] 
		& \Omega^p(\vec{G}_2)  \\
		\Omega^p(\vec{G}_1) \arrow[u, "\f j_1^*"]  \arrow[r,leftarrow, "\f i_1^*"]
		& \Omega^p(\vec G) \arrow[u, "\f i_2^*"]
	\end{tikzcd}
\end{equation}
for all $p\geq 0$. 
Let $\f i:\Omega^p(\vec G)\rightarrow \Omega^p(\vec G_1)\oplus \Omega^p(\vec G_2)$ and  $\f j:\Omega^p(\vec G_1)\oplus \Omega^p(\vec G_2)\rightarrow \Omega^p(\vec G_1 \cap \vec G_2)$ be the maps
\begin{align*}
\f i(\gamma) & =(\f i_1^*(\gamma),\f i_2^*(\gamma)), \\
\f j(\alpha,\beta) & =\f j_1^*(\alpha) - \f j_2^*(\beta).
\end{align*}

\begin{lem}\label{lem:ExactAtMiddle}
The sequence
$$
\begin{tikzcd}
	\Omega^1(\vec G) \ar[r,"\f i"] &  \Omega^1(\vec G_1)\oplus \Omega^1(\vec G_2) \ar[r,"\f j"] & \Omega^1(\vec  G_1 \cap \vec G_2) 
\end{tikzcd}
$$
is exact at $\Omega^1(\vec G_1)\oplus \Omega^1(\vec G_2)$.
\end{lem}
\begin{proof}
Recall from Remark~\ref{rmk:Omega_01} that $\Omega_1(\vec G) = \mathcal{A}_1(\vec G)$ is the free abelian group generated by the edges of $\vec G$.

It is clear that $\f j \circ \f i = 0$, and thus it suffices to show that $\ker(\f j) \subset \image(\f i)$. 
Suppose that $(\alpha,\beta)$ is an arbitrary element in $\ker(\f j)$, i.e.\ $\alpha \in \Omega^1(\vec G_1)$ and $ \beta \in  \Omega^1(\vec G_2)$ satisfy the equation $\f j_1^*(\alpha) = \f j_2^*(\beta)$ in $ \vec G_1 \cap \vec G_2$. 
    To show $(\alpha,\beta) \in \image(\f i)$, we define $\gamma\in\Omega^1(\vec G) = \Hom_\Z(\Omega_1(\vec G), \Z)$ to be the homomorphism such that 
\begin{equation*}
	\gamma(e_{i_0 i_1})=
	\begin{cases} 
		\alpha(e_{i_0 i_1}), & \text{if } e_{i_0 i_1} \text{ is an edge in $\vec G_1$,} \\
		\beta(e_{i_0 i_1}), & \text{if } e_{i_0 i_1} \text{ is an edge in $\vec G_2$.} 
	\end{cases}
\end{equation*}
Note that $\gamma$ is well-defined since if $e_{i_0 i_1}$ in both $\vec G_1$ and $\vec G_2$, then it follows from the assumption $(\alpha, \beta) \in \ker(\f j)$ that $\alpha(e_{i_0 i_1})=\beta(e_{i_0 i_1})$. 
Furthermore, by the definition of $\gamma$, we have $(\alpha,\beta) = \f i (\gamma) \in \image(\f i)$, as desired.
\end{proof}

\begin{lem}\label{lem:OntoAtEnd}
The map 
$$
\f j: \Omega^0(\vec G_1)\oplus \Omega^0(\vec G_2) \to \Omega^0(\vec G_1 \cap \vec G_2) 
$$
is surjective.
\end{lem}
\begin{proof}
Let $V$, $V_1$ and $V_2$ be the sets of vertices of $\vec G$, $\vec G_1$ and $\vec G_2$, respectively. Since $\vec G = \vec G_1 \cup \vec G_2$, we have 
$$
V = (V_1 \cap V_2) \coprod (V \setminus V_2) \coprod  (V \setminus V_1).
$$
By Remark~\ref{rmk:Omega_01}, we have 
$$
\Omega_0(\vec G_1) = \Omega_0(\vec G_1 \cap \vec G_2) \oplus A_1,
$$
where $A_1$ is the free abelian group generated by $V \setminus V_2$. The lemma follows immediately from this decomposition.
\end{proof}

Finally, we prove the main result of this section.

\begin{prp}\label{prop:H1Brown}
The functor $H^1:\hD_0^{\text{op}}\to \textbf{Ab}$ is a Brown functor.
\end{prp}
\begin{proof}
It is straightforward that $H^1$ satisfies the additivity axiom. For the Mayer-Vietoris axiom, first note that, due to the commutativity of the diagram \eqref{eq:InducedByInclusion}, we have the map 
$$
    H^1(\vec G_1 \cup \vec G_2) \to H^1(\vec G_1) \times_{H^1(\vec G_1 \cap \vec G_2)} H^1(\vec G_2), \quad [\gamma] \mapsto ( [\f i_1^\ast \gamma], [\f i_2^\ast \gamma]).
$$
We claim that it is surjective.

Consider the commutative diagram
$$
	\begin{tikzcd}
		 \Omega^1(\vec G_1 \cup \vec G_2)  \arrow[r, "\f i "]
		& \Omega^1(\vec G_1)\oplus \Omega^1(\vec G_2)  \arrow[r, "\f j "]
		& \Omega^1(\vec G_1\cap \vec G_2)  
		\\
		 \Omega^0(\vec G_1 \cup \vec G_2) \arrow[u, "d"]  \arrow[r, "\f i"]
		& \Omega^0(\vec G_1)\oplus \Omega^0(\vec G_2) \arrow[u, "d \oplus d"]  \arrow[r, "\f j"]
		& \Omega^0(\vec G_1\cap \vec G_2) \arrow[u, "d"] .
	\end{tikzcd}
$$
 Let $([\alpha],[\beta])\in H^1(G_1)\times H^1(G_2)$ be an arbitrary pair of classes such that $\f j_1^*([\alpha])=\f j_2^*([\beta])$, i.e.\ $\f j(\alpha,\beta) = \f j^*_1(\alpha) - \f j^*_2(\beta)=dx$ for some $x\in\Omega^0(\vec G_1\cap \vec G_2)$. By Lemma~\ref{lem:OntoAtEnd}, there exists $(y_1,y_2)\in\Omega^0(\vec G_1)\oplus \Omega^0(\vec G_2)$ such that $\f j(y_1,y_2)=x$. 

Consider  $(z_1,z_2)=(\alpha-dy_1,\beta-dy_2) \in \Omega^1(\vec G_1)\oplus \Omega^1(\vec G_2)$. Since 
$$
\f j(dy_1,dy_2)=dx= \f j^*_1\alpha - \f j^*_2\beta= \f j(\alpha,\beta),
$$
we have $(z_1,z_2) \in \ker(\f j)$. By Lemma~\ref{lem:ExactAtMiddle}, there exists $\gamma\in\Omega^1(\vec G_1\cup \vec G_2)$ such that $(\f i^*_1\gamma,\f i^*_2\gamma)=(z_1,z_2)$, i.e.\ 
$$
\f i^*_1\gamma=\alpha-dy_1 \quad \text{and} \quad \f i^*_2\gamma=\beta-dy_2.
$$
Thus $( [\f i_1^\ast \gamma], [\f i_2^\ast \gamma]) =([\alpha],[\beta])$, and the proof is complete.
\end{proof}

 Since the structure of bases is crucial in our proofs of Lemma~\ref{lem:ExactAtMiddle} and Lemma~\ref{lem:OntoAtEnd}, and we do not have good descriptions of bases for $\Omega_p(\vec{G})$ for $p > 1$, it is not clear to us at this moment whether $H^p$ for $p > 1$ is a Brown functor. (See \cite{BC24} for a discussion on bases for $\Omega_p(\vec{G})$.) In fact, due to similar difficulties, technical assumptions are imposed in \cite[Theorem~3.25]{GJMY18} to obtain Mayer-Vietoris sequences of path homology.

It is well-known that there are natural bijections between singular cohomology $H^n(X,\Z)$ and the homotopy classes $[X,K(\Z,n)]$ of continuous maps from $X$ to $K(\Z,n)$; see, for example, \cite[Theorem~4.57]{Ha}. This naturally raises the question of whether there is a similar bijection for path cohomology. There are two main difficulties. First, a theory of Eilenberg–MacLane spaces $K(\Z,n)$ is not yet well established for digraphs. Second, recall that \cite[Theorem~4.57]{Ha} based on the uniqueness theorem of the cohomology theory, but such a uniqueness theorem has not been established for digraphs; see \cite[Remark~5.3]{GJMY18}.

\end{document}